\documentclass[11pt,reqno]{amsart}

\usepackage{mathtools}
\usepackage{mathrsfs}
\usepackage{multicol}
\usepackage{fancyvrb}
\newcommand{\Vx}{%
\ooalign{\Large $\bigcup$\cr\hss\raisebox{-0.05ex}{\scriptsize $\cdot$}\hss}}%

\usepackage{amssymb,latexsym,amsbsy,amsthm,
amsxtra,amscd,amsfonts,amsthm,amsmath,verbatim}

\usepackage{algpseudocode,algorithmicx}
\usepackage{algorithm,caption}

\algrenewcommand\algorithmicrequire{\textbf{Initialize:}}
\algrenewcommand\algorithmicensure{\textbf{Output:}}

\newcommand{\ncom}{\newcommand}
\ncom{\bq}{\begin{equation}}
\ncom{\eq}{\end{equation}}
\ncom{\beqn}{\begin{eqnarray*}}
\ncom{\eeqn}{\end{eqnarray*}}
\ncom{\beq}{\begin{eqnarray}}
\ncom{\eeq}{\end{eqnarray}}
\ncom{\been}{\begin{enumerate}}
\ncom{\eeen}{\end{enumerate}}
\ncom{\olin}{\overline}
\ncom{\f}{\frac}
\ncom{\rar}{\rightarrow}

\usepackage{lmodern}
\usepackage[T1]{fontenc}

\usepackage[margin=1in]{geometry}
\usepackage{paralist}

\usepackage{color}

\usepackage[colorlinks,pagebackref=true]{hyperref}

\makeatletter

\@addtoreset{equation}{section}
\def\theequation{\thesection.\@arabic \c@equation}

\def\@citecolor{blue}
\def\@linkcolor{blue}
\def\@urlcolor{blue}

\makeatother

\usepackage{amssymb,latexsym,color,amsxtra,amscd,amsfonts,amsthm,amsmath,verbatim,epsfig}

\usepackage{multicol}
\usepackage{fancyvrb}
\usepackage{xypic}
\usepackage[all,cmtip]{xy}
\input xy
\xyoption{all}

\ncom{\limnn}{\underset{\underset{n}{\longrightarrow}}{\lim}}

\usepackage[colorlinks,pagebackref=true]{hyperref}
\def\@citecolor{blue}
\def\@linkcolor{blue}
\def\@urlcolor{blue}

\def\theequation{\thesection.\arabic{equation}}
\def\theequation{\arabic{equation}}
\numberwithin{equation}{section}

\def\ass{\operatorname{Ass}}
\def\deg{\operatorname{deg}}

\def\dim{\operatorname{dim}}

\def\height{\operatorname{ht}}

\def\ker{\operatorname{ker}}
\def\minass{\operatorname{MinAss}}
\def\Proj{\operatorname{Proj}}
\def\reg{\operatorname{reg}}
\def\supp{\operatorname{Supp}}

\def\lrar{{\longrightarrow}}

\def\A{\mathbb A}
\def\F{\mathcal F}
\def\I{\mathcal I}
\def\J{\mathcal J}

\def\N{\mathbb N}
\def\P{\mathbb P}
\def\R{\mathcal R}
\def\Z{\mathbb Z}

\newcommand{\m}{\mathfrak m}

\newcommand{\kk}{\Bbbk}
\newcommand{\n}{\mathfrak n}
\newcommand{\p}{\mathfrak p}
\newcommand{\q}{\mathfrak q}

\def\wt{\operatorname{wt}}

\theoremstyle{definition}
\newtheorem{theorem}[equation]{Theorem}
\newtheorem{corollary}[equation]{Corollary}
\newtheorem{proposition}[equation]{Proposition}
\newtheorem{lemma}[equation]{Lemma}

\theoremstyle{definition}
\newtheorem{notation}[equation]{Notation}

\newtheorem{definition}[equation]{Definition}

\ncom{\bib}{\bibitem}

\ncom{\limns}{\underset{\underset{s}{\longrightarrow}}{\lim}}
\ncom{\limnr}{\underset{\underset{r}{\longrightarrow}}{\lim}}

\ncom{\Tprime}{T^{\prime}}
\ncom{\mprime}{\m^{\prime}}

\ncom{\limm}{\underset{{ n \to \infty}}{\lim}}

\def\z{{\bf z}}
\def\x{{\bf x}}
\def\ff{{\bf f}}

\begin{document}
 \title[ ] {Symbolic Blowup algebras  and invariants of  certain monomial curves in an affine space}
 \author{Clare D'Cruz}
 \address{Chennai Mathematical Institute, Plot H1 SIPCOT IT Park, Siruseri, 
Kelambakkam 603103, Tamil Nadu, 
India
} 
\email{clare@cmi.ac.in} 
 \author[Masuti] {Shreedevi K. Masuti}
\address{Chennai Mathematical Institute, H1-SIPCOT IT Park, Siruseri, Kelambakkam - 603 103, India}
\email{shreedevikm@cmi.ac.in}

 \keywords{Symbolic  Rees algebra, Cohen-Macaulay, Gorenstein}
 \thanks{Both authors are partially funded by a grant from Infosys Foundation}
 \thanks{SKM is supported by INSPIRE faculty award funded by Department of Science and Technology, Govt. of India. }
 \subjclass[2010]{Primary: 13A30, 1305, 13H15, 13P10}

\begin{abstract}
Let $d \geq 2$ and  $m\geq 1$ be integers such that  $\gcd (d,m)=1.$  Let $\p$ be the defining ideal of the monomial curve in $\A_{\kk}^d$ parametrized by $(t^{n_1}, \ldots, t^{n_d})$ 
where $n_i = d + (i-1)m$ for all  $i  = 1, \ldots, d$. In this paper,  we   describe the symbolic powers 
$\p^{(n)} $ for all $n \geq 1$.  
As a consequence we  show that  the symbolic blowup  algebras $\R_s{(\p)}$ and   $G_{s}(\p) $  are Cohen-Macaulay.  
This gives a positive answer to a question posed by  S.~Goto in \cite{goto}.
We also discuss when these blowup algebras are Gorenstein. 
Moreover, for $d=3$, considering  $\p$ as a weighted homogeneous ideal, we compute  the resurgence,  the Waldschmidt constant and the  
 Castelnuovo-Mumford regularity of $\p^{(n)}$ for all $n \geq 1$. 
The techniques  of this paper for computing $\p^{(n)}$ are new and we hope that these will be useful to study the symbolic powers of other prime ideals.
\end{abstract}

\maketitle

\section{Introduction}
Let $I $ be an ideal in a Noetherian ring $A$.  Then for all $n \geq 1$, the  $n$-th symbolic power  of $I$  is  the ideal
$\displaystyle { I^{(n)} :=\cap_{\p \in \minass(A/I)} ( I A_{\p} \cap A} )$.    In this paper we are interested in the symbolic powers of certain  prime ideals in the polynomial ring and  power series ring.  In particular, let $T := \kk[x_1, \ldots, x_d]$  
and  $\p:= \p_{{\mathcal C}(n_1, \ldots, n_d)}  \subseteq  T$ be the  defining ideal of the monomial curve in  $\A_{\kk}^d$ paramtererized 
by $(t^{n_1}, t^{n_2}, \ldots, t^{n_d})$,  where $t \in \kk$ and  $n_i = d + (i-1)m$ for all  $i  = 1, \ldots, d$. 
If  $R:= k[[ x_1, \ldots, x_d]]$ denotes the $\m$-adic completion of $T$, where $\m = (x_1, \ldots, x_d)$,  then  $\p^{(n)}R  = \p^{(n)}T \otimes_T R$. 
The first part of our paper concerns the Cohen-Macaulay and Gorenstein  property of  the  symbolic Rees algebra $R_{s}(\p) := \oplus_{n \geq 0} (\p^{(n)}R )t^n$ and symbolic associated graded ring $G_s(\p) := \oplus_{\n \geq 0} \p^{(n) } R / \p^{(n+1)} R$ when   $\R_{s}(\p)$ is Noetherian. The second part of our paper concerns the computation of resurgence and Waldschmidt constant of  $\p T$. We also give a formula for the Castelnuovo-Mumford regularity of $\p^{(n)}T$ for all $n \geq 1$. 

The $n$-th symbolic power  $\p^{(n)}$ is of interest for several reasons. It is related to an open question which goes back to the work of 
L.~Kronecker  \cite{kronecker}, where he showed that  every irreducible curve in $\A_{\kk}^{d}$ can be defined by $(d+1)$-equations. 
In $1981$, R.~Cowsik gave a striking relationship between  the Noetherianness of  the symbolic Rees algebra and the problem of set-theoretic complete intersection. 
 He showed that if $\p$ is a prime ideal in a regular local ring $R$ such that 
$\dim(R/ \p)=1$ and 
$\R_{s}(\p)$ is Noetherian, then  $\p$ is a set-theoretic complete intersection 
\cite{cowsik}.
Motivated by 
Cowsik's result, 
 in 1987,  C.~Huneke gave necessary and sufficient conditions for 
$\R_{s}(\p)$ to be Noetherian when $\dim~ R=3$ \cite{huneke}. Huneke's result  was 
generalised  for $\dim~ R \geq 3$ by M.~Morales  \cite{morales}. 
In general, the symbolic Rees algebra need not be Noetherian even for an affine monomial curve in $\A^3$ and depends on the characteristic of $\kk$ \cite{gnw}.
This unpredictable behaviour attracted the attention of several researchers and properties of the Noetherian symbolic Rees algebra was studied in several cases (for example see \cite{eliahou}, \cite{huneke0}, \cite{goto-nishida-shimoda}, 
\cite{goto-nishida-shimoda2}, \cite{gnw}, \cite{herzog-ulirch}, \cite{kurano}, \cite{reed},  \cite{schenzel}, \cite{schenzel2}  and  \cite{vasconcelos2}).

The main difficulty  in the study of the symbolic Rees algebra is  describing  the generators of the symbolic powers. The symbolic powers $\p^{(2)}$ and $\p^{(3)}$ for monomial curves in $\A^3$ has been studied extensively (\cite{eliahou}, 
\cite{huneke0}, \cite{schenzel}, \cite{schenzel2} and \cite{vasconcelos2}). 
In fact,  using 
the ideas in \cite{vasconcelos1} and \cite{vasconcelos2},  J.~Herzog and B.~Ulrich  gave a characterization for 
  $\R_{s}(\p) = R[\p t, \p^{(2)} t^2]$ \cite[Corollary~2.12]{herzog-ulirch}. However, for  $d \geq  4$,  there are very few results on    $\p^{(n)}$, $n \geq 1$ (\cite{goto}    and \cite{scolan}).  
In   1994, S.~Goto gave necessary and sufficient conditions for the  Cohen-Macaulayness and Gorensteiness of the  symbolic  blowup algebras when 
 $\R_{s}(\p)$ is Noetherian \cite{goto}.
The Gorenstein property of $\R_{s}(\p)$  for monomial curves has also been studied in
\cite{simis-trung}, \cite{goto-nishida-shimoda} and   \cite{goto-nishida-shimoda2}.

In the last decade,  motivated by the work in \cite{eisenbud-mazur}, \cite{ein-laz-smith} and    \cite{hochster-huneke},   there has been a great interest in the  relation between the symbolic powers  and ordinary powers  of ideals.  Since symbolic powers are hard to describe,  in order to compare the ordinary and symbolic powers of  a homogenous ideal  $ I \subset T$,  
C.~Bocci and B.~Harbourne   defined an asymptotic quantity  called the resurgence of $I$ which is defined as  $\rho(I) = \sup \{ m/r : 
I^{(m)} \not \subset I^r\}$  \cite{BH}. They observed that it exists for radical ideals.  The  resurgence is hard to compute in general and is a challenging problem. Hence, in order to give a bound for  $\rho(I)$,   in the same paper  they  defined another invariant  $\gamma(I)$ 
called  it the Waldschmidt constant.   The Waldschmidt constant of $I,$ denoted as $\gamma(I)$, is defined as
${\displaystyle
\gamma(I) = \limm~\f{\alpha( I^{(n)})}{n},
}$
where $\alpha(I):= \min \{ n | I_n \not = 0 \}$.  They showed that if $I$ is a homogenous ideal, then $\alpha(I) / \gamma(I) \leq \rho(I)$ and in addition if $I$ defines a  zero dimensional subscheme  in a projective space, then
$ \rho(I) \leq \reg (I) / \gamma(I)$ where   $\reg(I)$ denotes the Castelnuovo-Mumford regularity of $I$ \cite[Theorem~1.2.1]{BH}.  
The resurgence  and the Waldschmidt constant has been studied in a few cases: for certain general points in $\P^2$  \cite{BH0},  smooth subschemes  \cite{guardo}, fat linear subspaces 
\cite{fatabbi},  special point configurations \cite{duminicki} and monomial ideals \cite{bocco-waldschmidt}.

We now describe the work in this paper. Let     $\gcd(d,m)=1$,  $n_i:= d + (i-1) m$ 
for  $i=1, \ldots,d$ and $\p:= \p_{{\mathcal C}(n_, \ldots, n_d)}  \subseteq  R = \kk[[x_1, \ldots, x_d]]$.  In 1994, S.~Goto showed that  $\R_{s}(\p)$ is Noetherian for all $d \geq 2$ and   is 
Cohen-Macaulay  if  $d \leq 4$  \cite[Proposition~7.6]{goto}. In the same article he  raised  the question  whether $\R_{s}(\p)$ is Cohen-Macaulay if $d\geq 5$  \cite[page 58]{goto}. 
In this paper we explicitly describe $\p^{(n)}$ for all $n \geq 1$. 
For this,  we first define the ideal $\mathcal{I}_n \subseteq \p^{(n)}$ (see \ref{equation of In}). One important observation is that  $ {\mathcal I}_n T^{\prime}$   is a homogenous 
ideal (Proposition~\ref{description of In}) where $T^{\prime} = T/ x_1 T$. 
The new idea of this paper is to give a monomial   order on  the monomials in $T^{\prime} = T/ (x_1)$ (Definition~\ref{monomial order}) and compute the leading ideal of $\p^{(n)}\Tprime$ (Theorem~\ref{symbolic power}).
More precisely, we  define monomial  ideals $I_n \subseteq T^{\prime}$ (\ref{ definition of In})  and show that  $I_n = LI( \p^{(n)} T^{\prime})$.
As a consequence, we show that   $\p^{(n)} R=  \mathcal{I}_n R$ for all $n \geq 1$.

As a first application, we show that $\R_s({\p})$ is Cohen-Macaulay for all $d \geq 2$ (Theorem~\ref{cm-rees}(\ref{cm-rees-one})).   This  gives a positive answer to Goto's question.  We also show  that   $G_s(\p) $ is  Gorenstein for all $d \geq 2$ (Theorem~\ref{cm-ass-gr}) and $\R_s({\p})$ is Gorenstein if and only if $d=3$ (Theorem~\ref{cm-rees}(\ref{cm-rees-two})). 

 We elucidate other applications in this paper.  Let $d=3$. We put weights  $wt(x_i) = n_i$ where $n_i:= 3 + (i-1) m$ and  $i=1, 2, 3$. With these weights, $\p^{(n)} = (\p^n)^{sat}$ defines a fat  point  for all $n \geq 1$  in the weighted projective space 
$\P := \P^{2}(n_1, n_2, n_3) := \Proj(T)$. Since $\p$ is a weighted homogenous ideal,  we extend  the definition of resurgence and Waldschmidt constant  to  $\p$. 
Moreover,   we observe that  Theorem~1.2.1 of \cite{BH} holds true for $\p$ (Theorem~\ref{lb for resurgence} and Theorem~\ref{ub for resurgence}). 

 In   \cite[Theorem~1.1]{cut-kurano} Cutkosky and Kurano showed that 
$\limnn  \reg(    (\p^{n})^{sat}   ) /n $ exists and  $\reg(T/\p^{(n)})$ is eventually periodic  \cite[Corollary~4.9]{cut-kurano}.  We give an explicit formula for  $ \reg(  T/  (\p^{n} )^{sat} )$ 
for $d=3$ and for  all $n \geq 1$. In particular ${ \displaystyle \limnn  \reg(    (\p^{n})^{sat}   ) /n=\frac{3 e(T/ \p)}{2}  + 3m}$ (Theorem~\ref{final regularity}).

We remark that using the techniques of this papers one can compute the resurgence, Waldschmidt constant and regularity of $\p^{(n)}$  for $d \geq 4$. However, this involves tedious computation and hence we restrict ourselves to $d=3$ in this paper. 

We now describe the organisation of this paper. In Section~2 we prove some preliminary results which will be needed in the subsequent sections. In Section~3 we describe the monomial order we are using and 
describe  the ideals $I_nT^{\prime} \subseteq LI( { \mathcal I}_n T^{\prime})$. Section~4 is mainly devoted to  show that the associated graded ring corresponding to the filtration $\{I_n\}_{n \geq 0}$ is Cohen-
Macaulay. In Section~5 we explicitly describe the monomials which span  $I_{n-1}$ modulo  $(I_n:x_d)$. In Section~6 we explicitly describe all the symbolic powers $\p^{(n)}$. 
The main results of this paper are in  Section~7.  
In this section we study the Cohen-Macaulay and Gorenstein property 
of $\R_s(\p)$ and $G_s(\p)$ for $d \geq 2$, and give an explicit formula for the resurgence and Waldschmidt constant. Moreover, we also compute the Castelnuovo-Mumford regularity of the symbolic powers.  

\section*{Acknowledgements} The authors would like to thank Prof. J.~K.~Verma for suggesting the problem and for many useful conversations. The second  author would like to thank the Department of Atomic Energy, Government of India, and the Chennai Mathematical Institute (CMI) for providing financial support for her post doctoral studies at the Institute of Mathematical Sciences (IMSc) and CMI, respectively, during which part of the work is done. She also thanks INdAM cofunded by Marie Curie actions, Italy, for her research in Genova during which part of the work is done

 \section{Preliminaries}
 In this paper,   we consider the following class of monomial curves: Let $d \geq 2$. Let  $R= \kk[[x_1, \ldots, x_d]]$ and  $S = \kk[[t]]$ be formal 
power series rings over  $\kk$. For any positive integer $m \geq 1$,  with $\gcd(d,m)=1$, we put $n_i:= d + (i-1) m$ for 
$i=1, \ldots,d$.
Let ${\mathcal C}(n_1, \ldots, n_d)$ be the affine curve parameterised by $(t^{n_1}, \ldots, t^{n_d})$ and let   $I_{{\mathcal 
C}(n_1, \ldots, n_d)}$  be the ideal defining this monomial curve. In other words,  let $\phi: R \lrar \kk[[t]]$ denote the 
homomorphism defined by $\phi(x_i) = t^{n_i}$ for $1 \leq i \leq d$  and  $\p:= \ker(\phi)= I_{{\mathcal C}(n_1, \ldots, n_d)}$. 

Throughout this paper   $\p=I_{{\mathcal C}(n_1, \ldots, n_d)}$ unless otherwise specified.
 It is well known that $\p$ is generated  by the $2 \times 2$ minors of the matrix  described in \eqref{matrix of p}. In 
\cite[Proposition~7.6]{goto}, Goto described $\p^{(n)}$ for $d=4$ and  $n=2,3$. It is not easy to describe the ideals $\p^{(n)}$  in general. 
 To achieve this,  we define ideals  ${\mathcal I}_n R\subseteq \p^{(n)}$ (see \eqref{equation of In}). 
We exploit the fact that the ideals  $ ({\mathcal I}_n , x_1) T$    are homogeneous ideals (Proposition~\ref{description of In}).  
\subsection{Computation of multiplicity}
\label{section 2}
Let $R = \kk[[x_1, \ldots, x_d]]$ and  $X = [X_{ij}]$ be the $d \times d$ matrix given by 
\beq
\label{matrix of p}
X_{ij}:= \left\{ 
\begin{array}{ll}
x_{i+j-1} & \mbox{ if } 1 \leq i \leq d \mbox{ and }  1 \leq j \leq d-i+1\\
x_1^m x_{i+j-d-1} & \mbox{ if } 2 \leq i \leq d \mbox{ and }d-i+2 \leq j \leq d.\\
\end{array}
\right. 
\eeq
For each $1 \leq i,k \leq d-1$,   we define:
\beq
\label{definition of X_i}
X(i) &:=& \mbox{The matrix consisting of the first $i+1$ rows and $i+1$ columns of $X$}, \\
\label{definition of f_i}
f_i &:=&  \det(X(i)  )         \hspace{.2in} \mbox{ and }  
\ff_k := f_1,\ldots, f_k.
\eeq
 Goto showed that  $\ff_{d-1}$ satisfies  Huneke's criterion for the Noetherianness of 
 $\R_{s}({\p})$ (\cite[Theorem~7.4]{goto}).

In this section we give a lower bound for the length of the modules 
$ R/ (\p^{(n)} +  ( x_1, \ff_k ) )$ where  $\p= I_{{\mathcal C}(n_1, \ldots, n_d)},$ $1\leq k \leq d-1$ and $n \geq 1.$ We need  a few preliminary results.

Let $(A, \n)$ be a Noetherian local ring of positive dimension $d$ and $\mathfrak{a}$ an $\n$-primary ideal.  
Let $\F=\{\F(n)\}_{n \in \Z}$ be a Noetherian filtration  of ideals, i.e., $\F(0) = A$, $\F(1) \not =A $, $\F(n+1) \subseteq \F(n)$, 
$\F(n) \cdot \F(m) \subseteq \F(n+m)$ for all $n,m \in \Z$ and the Rees ring $\mathcal{R}(\F) := \oplus_{n \geq 0} \F(n) t^n$ is 
Noetherian.
Let  $1   \leq k \leq d$ and   $z_i \in {\F}(a_i) \setminus {\F}(a_i +1)$  for all $i=1, \ldots
k$. Put $\z_k = z_1, \ldots, z_k$. For all $n \in \Z$,
using the mapping cone construction, similar to that in  \cite{huckaba-marley}, we  construct the complex
$C_{\bullet}(\z_k; n)$ which has the form:
\beq
\label{complex}
    0 
\lrar \f{A}{\F({n- (a_1  + \ldots + a_k)})}
\lrar \cdots
\lrar \bigoplus_{1\leq i < j \leq k} \f{A}{{\F}({n- a_i-a_j})}
\lrar \bigoplus_{i=1}^k \f{A}{{\F}({n- a_i})}
\lrar \f{A}{{\F}(n)}
\lrar 0.
\eeq
The maps are  from the Koszul complex $K_{\bullet} (\z_k, A)$.
Let $H_i (  C_{\bullet}( \z_k, n))$ denote the $i$-th homology  of the complex $C_{\bullet}(\z_k; n)$. 

For any element $z \in {\F}(n) \setminus {\F}({n+1})$, let $z^{\star}$ denote the image of $z$ in
 $G(\F):=\oplus_{n \in \N}{\F(n)}/{\F(n+1)}$. 
Let  $\z_k^{\star}:= z_1^{\star},\cdots, z_k^{\star}$.

\begin{proposition}
\label{vanishing}
Let $\{{\F}(n)\}_{n \geq 0}$ be a filtration of $\m$-primary ideals.
For $1 \leq i \leq k$, 
let $z_i \in {\F}({a_i}) \backslash  {\F}({a_i+1})$.
Suppose $\z_k^{\star}$ is  a regular sequence  in $G({\F})$. Then
\been
\item
\label{vanishing one}
  $H_i (  C_{\bullet}( \z_k, n))=0$ for all $i \geq 1$ and all $n \in \Z$. 
  \item
  \label{vanishing two}
  $
  {\displaystyle
  \ell \left( \f{A} {{\F}(n) +  (\z_k) }\right) 
  = \sum_{i=0}^{k} 
(-1)^{i} \left[ \sum_{1 \leq j_1 < \cdots <j_i  \leq k }  
\ell
\left( \f{A}{{\F}({n - [a_{j_1} +  \cdots + a_{j_i} ]})} \right) \right].}
  $
\eeen
\end{proposition}
\begin{proof} (\ref{vanishing one})
Let $K_{\bullet} (\z_k^{\star}, G({\F}))$ denote the Koszul complex of $G({\F})$ with respect to $\z_k^{\star}$.  Then we have the short exact sequence of complexes:
\beq
\label{ses}
 0 \lrar K_{\bullet} (\z_k^{\star}, G({\F}))_{n-1} \lrar  C_{\bullet}( \z_k, n) \lrar  C_{\bullet}( \z_k, n-1) \lrar 0. 
\eeq
Since $\z_k^{\star}$ is a regular sequence in $G(\F)$, $H_i( K_{\bullet} (\z_k^{\star}, G({\F})))=0$ for all $i \geq 1$ \cite[Theorem~16.5]{matsumura}. Hence from \eqref{ses}
for all $n \in \Z$ we have:
\beqn
H_i (  C_{\bullet}( \z_k, n)) \cong H_i (  C_{\bullet}( \z_k, n-1)) \hspace{.2in} \mbox{for all } i \geq 2
\eeqn
and the short exact sequence
\beqn
0 \lrar H_1(  C_{\bullet}( \z_k, n)) \lrar H_1(  C_{\bullet}( \z_k, n-1)). 
\eeqn
As $H_i (  C_{\bullet}( \z_k, n)) =0$ for all $n \leq 0$,  we conclude that  $H_i (  C_{\bullet}( \z_k, n)) =0$ for all $n$ and for all $i \geq 1$.
This proves (\ref{vanishing one}). 

 (\ref{vanishing two})  
 As $H_0 (  C_{\bullet}( \z_k, n)) = A/ ({\F}(n) + (\z_k))$, from the complex (\ref{complex})
 we get
 \beqn
 \ell \left( \f{A}{{\F}(n) + (\z_k)} \right)
  + \sum_{i \geq 1} (-1)^i~\ell (H_i (  C_{\bullet}( \z_k, n)))
 =   \sum_{i=0}^{k} 
(-1)^{i}\left[ \sum_{1 \leq j_1 < \cdots <j_i  \leq k }  
\ell
\left( \f{A}{{\F}({n - [a_{j_1} +  \cdots + a_{j_i} ]})} \right) \right].
 \eeqn 
 Applying (\ref{vanishing one}) we get the result.
\end{proof}

\begin{corollary}
\label{multiplicity}
Let $(A,\n)$ be a Cohen-Macaulay local ring of dimension $d$. Let $\p$ be a prime ideal of height $d-1$ and $x \notin \p.$ Let $1 \leq k \leq d-1$ and $z_i \in \p^{(a_i)} \backslash \p^{(a_i+1)}$. Suppose $\z_k^{\star}$ is a regular sequence  in $G(\p A_{\p})$. Then
\been
\item 
\label{multiplicity-one}
$
{\displaystyle 
    e\left(x; \f{A}{\p^{(n)} + (\z_k)} \right) 
= \ell \left( \f{A}{(\p,x)}\right)  
   \sum_{i=0}^{k} (-1)^{i} 
   \left[ \sum_{1 \leq j_1 < \cdots <j_i \leq k}  
   \ell
   \left( \f{A_{\p}}
             {\p^{n - [a_{j_1} +  \cdots +a_{j_i} ]} A_{\p}}
   \right)
   \right].}
$

\item
\label{multiplicity-two}
$
{\displaystyle 
\ell \left( \f{A} { \p^{(n)} + (\z_k) +( x) } \right)
\geq \ell \left( \f{A}{(\p,x)}\right)  
 \sum_{i=0}^{k} (-1)^{i} 
  \left[ \sum_{1 \leq j_1 < \cdots <j_i \leq k }  
\ell
\left( \f{A_{\p}}{\p^{n - [a_{j_1} +  \cdots +a_{j_i} ]} A_{\p}}
\right)
\right].
}
$
\eeen
\end{corollary}
\begin{proof}
(\ref{multiplicity-one}) 
As $\p^{(n)} \subseteq \p^{(n)} + (\z_k) \subseteq \p$, taking radicals we get $\sqrt{ \p^{(n)} + (\z_k)} = \p$. Hence $\p$ is the only minimal prime of  $\p^{(n)} + (\z_k)$. 
From the associativity   formula  for multiplicities \cite[Theorem~14.7]{matsumura} we get
\beqn
  e\left(x; \f{A}{\p^{(n)} + (\z_k)} \right) 
=e \left(x; \f{A}{\p} \right)
   \ell\left(  \f{A_{\p}} { (\p^{(n)} + (\z_k)) A_{\p}} \right).
\eeqn
As $x$ is a nonzero divisor on $A/ \p$, $e(x; A/ \p) = \ell (A/ (\p,x))$.  Replacing   $A$ by 
$A_{\p}$  and $G(\F)$ by $G(\p A_{\p})$ in Proposition~\ref{vanishing}(\ref{vanishing two}) we get the result.

(\ref{multiplicity-two}) From \cite[Theorem~14.10]{matsumura},  we get
$
{\displaystyle 
\ell\left(  
\f{A} { \p^{(n)} + (\z_k) + (x) } \right) 
\geq e\left(x; \f{A}{\p^{(n)} + (\z_k)} \right)}$. Now apply (\ref{multiplicity-one}).
\end{proof}

\begin{theorem}
\label{computation of multiplicity}
Let $R = \kk[[x_1, \ldots, x_d]]$ and $\p= I_{{\mathcal C}(n_1, \ldots, n_d)}$. For $1 \leq i,k \leq d-1$, let  $f_i$  and 
$\ff_k$ be as in (\ref{definition of f_i}).   Then 
\beqn
\ell \left( \f{R} {\p^{(n)} + (\ff_k) + (x_1) }\right) 
\geq  \ell \left( \f{R}{(\p,x)}\right)  
 \sum_{i=0}^{k} (-1)^{i} 
  \left[ \sum_{1 \leq j_1 < \cdots< j_i \leq k }  
\ell
\left( \f{R_{\p}}{\p^{n - [{j_1} +  \cdots +{j_i} ]} R_{\p}}
\right)
\right].
\eeqn
\begin{proof}
By \cite[Lemma~7.5]{goto}, $f_i \in \p^{(i)}$. 
As $G(\p R_{\p})$ is a regular ring  
and  $\ff_{d-1}^{\star}$ is a regular sequence \cite[Proposition~5.3(3)]{goto}, from Corollary~\ref{multiplicity}(\ref{multiplicity-two}), we get the result. 
\end{proof}
\end{theorem}

\subsection{The power series ring and the polynomial ring}
From now on  $R= \kk[[x_1, \ldots, x_{d}]]$ and  $T=\kk[x_1, \ldots,x_d]$.  The following lemma gives us a way to compute the length of an $R$-module in terms of the  length of the corresponding $T$-module.

 \begin{lemma}
 \label{comparing lengths}
 Let
  $\m = (x_1, \ldots, x_d)T$
  and $M$ a finitely generated  $T$-module such that  $\supp(M) = \{ \m\}$.  Then 
 \beqn
 \ell_R ( M \otimes_T R) = \ell_T(M).
 \eeqn
  \end{lemma}
 \begin{proof}
We prove  by induction on $\ell_T(M)$.  If $\ell_T(M)=1$, then $M \cong T/ \m$. Therefore,
\beqn
\ell_R(M \otimes_T R) 
= \ell_R \left(\f{R}{\m R} \right) = 1 \hspace{.1in} \mbox{(as $\m R$ is the maximal ideal of $R$)}. 
\eeqn
 
 If  $\ell_T(M)> 1$, then as the minimal primes of $\supp(M)$ and $\ass(M)$ are the same, $\m \in \ass(M)$.
This gives the exact sequence
 \beq
 \label{ses-1}
            0
\lrar \f{T}{\m}
\lrar M 
\lrar C 
\lrar 0,
 \eeq
 where $C \cong  M / (T/ \m)$.
 As $R$ is $T$-flat, tensoring (\ref{ses-1}) with $R$ we get:
  \beqn
            0
\lrar \f{T}{\m} \otimes_T R \cong \f{R}{\m R}
\lrar M \otimes_T R
\lrar C  \otimes_T R
\lrar 0.
 \eeqn
From the exact sequence (\ref{ses-1}), we get  $\supp(C) = \{\m\}$ and $\ell_T(C)  < \ell_T(M)$. Therefore by induction hypothesis
 $\ell_R(C  \otimes_T R) = \ell_T(C)$. Hence
 \beqn
   \ell_R( M \otimes_T R) 
= \ell_R(C  \otimes_T R) +  \ell_R \left( \f{R}{\m R}\right)
= \ell_T(C) + \ell_T\left( \f{T}{\m} \right)
= \ell_T(M).
 \eeqn
 \end{proof}

Let 
\beq
 \nonumber
&& X_{i+1, (j_1, \ldots, j_{i+1})} \\ \label{definition of X_i+1}
&:=& \mbox{the matrix obtained by choosing the first $i+1$ rows and  $j_1, \ldots, j_{i+1}$  columns of $X$}
\\
\label{equation of Jn} \nonumber
&& {\mathcal J}_i \\
 &:=& \{ \det( X_{i+1, (j_1, \ldots, j_{i+1}) } )| 1 \leq j_1 < \cdots < j_{i+1} \leq d \} .
\eeq

\begin{notation}
If $A_1, \ldots, A_n$  are $n$ sets of monomials we define the set $A_1 \cdots A_n$ by
$A_1\cdots A_n := \{a_1 \cdots a_n : a_i \in A_i \}$.
\end{notation}
 Let ${\mathcal I}_n$ denote the set 
\beq
\label{equation of In}
        {\mathcal I}_n 
:=   \sum_{a_1 + 2 a_2 + \cdots +  (d-1)  a_{d-1} = n} 
      \J_{1}^{a_1} \cdots \J_{d-1}^{a_{d-1}}.
\eeq
As $R$ is a flat $T$-module,
  $ {\mathcal I}_nR = {\mathcal I}_n T \otimes_T  R$.

\begin{proposition}
\label{description of In}
Let $n \geq 1$. Then
\been
\item
\label{description  of In one}
$    {\mathcal I}_n  R\subseteq   \p^{(n)}$.

\item
\label{description of In two}
$  ({\mathcal I}_n  + (x_1))T$  is a homogeneous ideal.  

\item
\label{description of In three}
 $({\mathcal I}_n  + (x_1))T$ is an $\m$-primary ideal.
\eeen
\end{proposition}
\begin{proof}
(\ref{description of In one})
By \cite[Lemma~7.5]{goto},  ${\mathcal J}_i \subseteq \p^{(i)}$ for all $i=1, \ldots, d-1$. Hence for all $a_1, \ldots, a_{d-1} \in \Z_{\geq 0}$, 
\beq
\label{containment of J}
 \J_{1}^{a_1} \cdots \J_{d-1}^{a_{d-1}}
 \subseteq \p^{a_1} (\p^{(2)})^{a_2} \cdots (\p^{(d-1)})^{a_{d-1}}
 \subseteq \p^{(a_1 + 2a_2 + \cdots + (d-1) a_{d-1})}.
\eeq
Summing over all $a_1+ 2a_2\cdots + (d-1)a_{d-1} =n$ and applying (\ref{containment of J})  to (\ref{equation of In}) we get (\ref{description of In one}).

(\ref{description of In two})  Fix $1 \leq j_1< j_2<\ldots < j_{i+1} \leq d$.  
Then $\det( X_{i+1, (j_1, \ldots, j_{i+1}) })$ is  a sum of distinct monomials and the monomials which do not contain $x_1$ are homogeneous of degree $i+1$. Hence $({\mathcal J}_i + (x_1))T$ is a homogeneous ideal. 
From (\ref{equation of In}) we get (\ref{description  of In two}). 

(\ref{description  of In three}) By \eqref{equation of In}, ${\mathcal J}_1^{n} \subseteq {\mathcal I}_n$ and  ${\mathcal J}_1^{n} + (x_1) = (x_2, \ldots, x_d)^{2n} + (x_1)$ which implies that 
$\m = \sqrt{{\mathcal J}_1^{n} + (x_1)} \subseteq \sqrt{{\mathcal I}_n + (x_1)} \subseteq \m$. 
\end{proof}

\section{Monomial ordering  and Initial ideals}
Using the description of $\p^{(n)}$ for $d=4$ and  $n=2,3$,   Goto proved that    the rings $R/ (\p^{(n)} +(f_1, f_2, f_3))$ are  Cohen-Macaulay, where the $f_i \in \p^{(i)}$  ($1 \leq i \leq 3$) are as described in \cite[page 57]{goto}. However, from their method it is not easy to prove a similar result for $d \geq 5$. 
The  new idea in this paper is to give an ordering on $T^{\prime} = T/ (x_1)$  which we call the grevelex  which is described in  Definition~\ref{monomial order}. 

\begin{definition}
\label{monomial order}
 Let ${\bf a } = (a_2, \ldots, a_d)$ and ${\displaystyle {\bf x}^{\bf a}:=  \prod_{i=2}^d  x_i^{a_i} }$.  
We say that 
${\bf x}^{\bf a}  >  {\bf x}^{\bf b}$ if  
${ \deg( {\bf  x}^{\bf a} ) > \deg(  {\bf x}^{\bf b} )}$ or 
$\deg( {\bf  x}^{\bf a} ) = \deg( {\bf x}^{\bf b} )$ and 
in the ordered tuple $(a_2-b_2, \ldots, a_d-b_d)$
the left-most nonzero entry is negative.
\end{definition}

 Note that with respect to this order we have $ x_2< x_3< \ldots < x_d$. For any polynomial 
$f \in T^{\prime}$, let $LM(f)$ denote the initial term of $f$ and 
for any ideal  $I \subset T^{\prime}$,  let  $LI(I)$ be the initial ideal  of the ideal $I$ with respect to 
the grevelex order. 
We 
define monomial ideals  $I_n\subseteq LI( {\mathcal I}_{n} T^{\prime})$ (\eqref{ definition of In}, Proposition~\ref{description 
of In in S}) and  consider the filtration  $\F = \{I_n\}_{\n \geq 0}$. In fact
$I_n = LI( {\mathcal I}_nT^{\prime})$ (Theorem~\ref{symbolic power}(\ref{symbolic power three})). 

For $2 \leq r < s \leq d$ and $l \geq 1$, let $M_{r,s}^l$ denote the set of monomials of degree $l$ in the variables $x_r, \ldots, x_s$. We set 
$M_{r,s}:=M_{r,s}^1.$ 

Let $1 \leq i \leq d-1$ and $n \geq 1$. We define the ideals $J_i$ and $I_n$ in $T^{\prime}$ as follows:
\beq
\label{definition of Ji}
J_i &:=&  (M_{i+1,d})^{i+1},\\
\label{ definition of In}
I_n &:= & \sum_{a_1 + 2 a_2 + \cdots  +(d-1)  a_{d-1} = n} J_{1}^{a_1} \cdots J_{d-1}^{a_{d-1} }.
\eeq

\begin{proposition}
\label{description of In in S}
For all $n \geq 1$, 
$I_n  \subseteq LI( {\mathcal I}_n \Tprime)$.
\end{proposition}

To prove  Proposition~\ref{description of In in S}, we first need to consider 
 $LM( \det( X_{i+1, (j_1, \ldots, j_{i+1}) } ) )$ for all $1 \leq j_1 < \cdots < j_{i+1} \leq d $. This is done in Proposition~\ref{proposition ji}.
 
\begin{notation}
For any $n \times n $ matrix $M = (m_{ij})$,  let $p(M):=\prod_{i+j = n+1} m_{ij}$ denote the product of anti-diagonal elements of the matrix $M$. 
\end{notation}

\begin{proposition}
\label{proposition ji}
For  $1 \leq i \leq d$, 
\been
\item
\label{proposition ji one}
$ {\displaystyle
p(X_{i+1, (j_1, \ldots, j_{i+1}) }) 
= LM (\det (X_{i+1, (j_1, \ldots, j_{i+1} ) } ) \Tprime)
= \prod_{k=1}^{i+1} x_{ j_k + (i-k+1) } 
.} 
$

\item
\label{proposition ji two}
$J_i   \subseteq LI ( {\mathcal J}_i \Tprime )$.
\eeen
\end{proposition}
\begin{proof} (\ref{proposition ji one}) By definition,  
$p (X_{i+1, (j_1, \ldots, j_{i+1}) } ) 
= \prod_{k=1}^{i+1} X_{i-k+2, j_k}$. 

 We claim that $X_{i-k+2, j_k}=x_{j_ k + (i-k+1)}$ for all $k=1, \ldots, i+1$. 
  Since $j_k \leq j_{i+1}-(i-k+1)$ for all $1\leq k\leq i+1,$ it follows that $j_k+(i-k+2)\leq j_{i+1} + 1 \leq  d+1.$ 
 Hence the matrix $X_{i+1, (j_1, \ldots, j_{i+1}) }$ is 
   \beq
    \label{matrix x(i+1)}
     \left(\begin{array}{ccccccc}
    x_{j_1}        
          & \cdots & x_{j_k}       & \cdots 
           & X_{1,j_{i+1}}=x_{j_{i+1}} \\
    \vdots 
    & \reflectbox{$\ddots$} & \reflectbox{$\ddots$}        & \reflectbox{$\ddots$} 
    & \vdots\\
    x_{ j_1 +( i-k+1)} 
& \reflectbox{$\ddots$} &X_{i-k+2, j_k}=x_{j_ k +( i-k+1)}       & \reflectbox{$\ddots$} 
 &\vdots \\
      \vdots 
      & \reflectbox{$\ddots$} & \reflectbox{$\ddots$}        & \reflectbox{$\ddots$}
       & \vdots\\
  X_{i+1, j_1}=x_{ j_1 + i}  
   & \reflectbox{$\ddots$} & \reflectbox{$\ddots$}       & \reflectbox{$\ddots$}  
   & \star \\
      \end{array}\right).
\eeq
This proves the claim. Hence
${\displaystyle
p (X_{i+1, (j_1, \ldots, j_{i+1}) } ) 
  = \prod_{k=1}^{i+1} x_{j_k + (i-k+1) }.}
$  
    
 To complete the proof of (\ref{proposition ji one}) we need to show that: 
\beq
\label{leading monomial}
LM(\det( X_{i+1, (j_1, \ldots, j_{i+1}) } )\Tprime )
= \prod_{k=1}^{i+1} x_{ j_k + (i-k+1)}.
\eeq  
We prove (\ref{leading monomial}) by induction on $i$.  Note that in the  matrix $X$ defined in (\ref{matrix of p}),   $x_1$ divides  $X_{11}$ and $X_{ij}$ for  $i+j \geq d+2$.  Let  $i=1$. Then
\beqn
\det( X_{2,  (j_1, j_2) } )\Tprime 
= \begin{cases}
 x_{j_1 +1} x_{j_2} & \mbox {if $j_1=1$ or $j_2=d$}\\
 x_{j_1} x_{j_2+1} - x_{j_1+1}x_{j_2}  & \mbox{ if $1< j_1 <j_2 <d$ }. 
\end{cases}
\eeqn
Hence $LM(\det( X_{2,  (j_1, j_2) } )\Tprime) =   x_{j_1 +1} x_{j_2}$ if $j_1=1$ or $j_2=d$. 
 If $1 < j_1 < j_2 <d$, then $j_1 < j_1 +1 \leq j_2 < j_2 +1$ and hence  $x_{j_1+1}x_{j_2} > x_{j_1} x_{j_2+1}$ which implies that  $LM(\det( X_{2,  (j_1, j_2) } )\Tprime) =  x_{j_1+1}x_{j_2}$. Hence (\ref{leading monomial}) is true for $i=1$.

Now let $i>1$. Expanding the matrix in  (\ref{matrix x(i+1)}) along the last row we get:
\beq
\label{expansion of det}
       \det(  X_{i+1, (j_1, \ldots, j_{i+1}) } )\Tprime
&=& \left( \sum_{k=1}^{t} (-1)^{k + i+1} x_{j_k + i}  \det (X_{i,  (j_1, \ldots, \widehat{j_k},  \ldots, j_{i+1})} )\right)\Tprime
\eeq
where $t = \max\{k | j_k + i \leq d \}$. 
As  $X_{i,  (j_1, \ldots, \widehat{j_k},  \ldots, j_{i+1} )} $ has the  form as the matrix described in $\eqref{matrix x(i+1)}$,
by induction hypothesis,  
\beq
\label{expansion of det k} \nonumber
LM (  \det(  X_{i,  (j_1, \ldots, \widehat{j_k},  \ldots, j_{i+1})})\Tprime  ) 
&=&  \prod_{\alpha =1}^{k-1} x_{j_{\alpha} + i - \alpha} \prod_{\alpha = k+1}^{i+1} x_{j_{\alpha} +( i - \alpha + 1 )}\\
&=& \begin{cases} \label{expansion of det k cases}\displaystyle
 \prod_{\alpha = 2}^{i+1} x_{j_{\alpha} +( i - \alpha + 1 )}  & \mbox{ if } k=1\\ 
 \displaystyle x_{j_1 + (i-1)}\prod_{\alpha =2}^{k-1} x_{j_{\alpha} + i - \alpha} \prod_{\alpha = k+1}^{i+1} x_{j_{\alpha} +( i - \alpha + 1 )} &  \mbox{ if } k =2, \ldots,t.
\end{cases}
\eeq
Hence  for all $k=2, \ldots, t$
\beq
\label{comparing leading terms}
 \nonumber 
 &&   x_{j_k + i} LM ( \det( X_{i,  (j_1, \ldots, \widehat{j_k},  \ldots, j_{i+1}) } ) \Tprime)\\ \nonumber
 &=& {\displaystyle x_{j_1 + (i-1)} \left[ \prod_{\alpha =2}^{k-1}x_{j_{\alpha} + i - \alpha} \right] x_{j_k+i}
          \left[ \prod_{\alpha = k+1}^{i+1} x_{j_{\alpha} + (i - \alpha + 1 )}\right]}
          \hspace{.5in}
           \mbox{[by \eqref{expansion of det k cases}]}\\
\nonumber          &<& {\displaystyle x_{j_1 +  i}  x_{j_2 +  (i-1)} \cdots x_{j_{i} +1}  x_{j_{i+1}}}\\
          &=& {\displaystyle x_{j_1 + i}LM ( \det( X_{i,  (\widehat{j_1},  j_2,   \ldots, j_{i+1}) } ) \Tprime ) }
          \hspace{1.62in} 
          \mbox{[by \eqref{expansion of det k}]}.
\eeq
Therefore  
\beqn
  LM(\det(  X_{i+1, (j_1, \ldots, j_{i+1}) } )\Tprime )
  &=& {\displaystyle  LM (  \sum_{k=1}^{t} (-1)^{k + i+1} x_{j_k + i} 
   LM (\det(  X_{i, (j_1, \ldots, \widehat{j_k},  \ldots, j_{i+1} ) }) )\Tprime ) }\hspace{0.2in} 
   \mbox{[by \eqref{expansion of det}]} \\
  &=& x_{j_1 + i}LM ( \det( X_{i,  (\widehat{j_1},  j_2,   \ldots, j_{i+1}) } ) \Tprime )
   \hspace{1.6in}
   \mbox{[by \eqref{comparing leading terms}] }\\
  &=& x_{j_1 +  i}  x_{j_2 +  (i-1)} \cdots x_{j_{i} +1}  x_{j_{i+1}}  \hspace{1.92in}
  \mbox{[by \eqref{expansion of det k cases}]}.  
\eeqn
This proves (\ref{proposition ji one}).

(\ref{proposition ji two}) Let $x_{i+k_1}^{\alpha_{i+k_1}} \cdots x_{i+k_s}^{\alpha_{i+k_s}}  \in M_{i+1,d}^{i+1}$ where 
$1 \leq k_1  < k_2 < \cdots < k_s\leq d-i $ and ${\alpha_{i+k_1}}, \ldots, {\alpha_{i+k_s}} \neq 0$ such that $\alpha_{i+k_1}+\cdots+\alpha_{i+k_s}=i+1.$ Set $\beta_0=0$ and $\beta_r=\alpha_{i+k_1}+\cdots+\alpha_{i+k_{r}}$ for $1\leq r \leq s.$ Define 
\beqn
S_r=\{\beta_{r-1}+1,\beta_{r-1}+2,\ldots,\beta_{r}\} \mbox{ for }1 \leq r \leq s.
\eeqn

Then 
$\bigsqcup_{r=1}^s S_r = \{1, \ldots, i+1\}$. Let $1 \leq t \leq i+1$. If $t \in S_r$ define 
\beqn
j_t & = k_r + (t-1).
\eeqn
With this choice of $j_1, \ldots, j_{i+1}$, $p( X_{i+1, (j_1, \ldots, j_{i+1})}  ) = x_{i+k_1}^{\alpha_{i+k_1}} \cdots x_{i+k_s}^{\alpha_{i+k_s}} $.  By (\ref{proposition ji one}) $x_{i+k_1}^{\alpha_{i+k_1}} \cdots x_{i+k_s}^{\alpha_{i+k_s}} \in LI ( {\mathcal J}_i \Tprime )$. Hence 
 $J_i \subseteq LI ( {\mathcal J}_i \Tprime )$.
\end{proof}

\noindent
{\bf Proof of Proposition~\ref{description of In in S}:}  The proof follows from (\ref{equation of In}), Proposition~\ref{proposition ji}\eqref{proposition ji two}  and (\ref{ definition of In}).

\section{The associated graded ring corresponding to the filtration $\F:= \{ I_n\}_{n \geq 0}$}
Let  $G(\F) :=\oplus_{n \geq 0} I_n /I_{n+1}$ be the associated graded  ring corresponding to the filtration $\F= \{ I_n\}_{n \geq 0}$, where $I_n$ are ideals defined in \eqref{ definition of In}.  
By definition of $I_n$,  $G(\F) $ is Noetherian  (Theorem~\ref{cohen macaulayness of G}). 
One of the key steps is to prove that the  associated 
graded ring $G(\F) = \oplus_{n \geq 0} (I_n/I_{n+1})$  is Cohen-Macaulay. In particular we show that  $(x_2^2)^{\star}, \ldots, (x_d^d)^{\star}$ is a regular sequence in  $G(\F)$  
(Theorem~\ref{cohen macaulayness of G}). 
 As an immediate  consequence, we give a formula for 
 ${\displaystyle \ell\left( \f{\Tprime}{(I_n + (x_2^2,  \cdots, x_{k}^{k} ))\Tprime}\right)}$ (Proposition~\ref{computing full length}) which is useful in the subsequent sections. The  following proposition is crucial to prove Theorem~\ref{cohen macaulayness of G}.

\begin{proposition}
\label{symbolic ideal quotients}
For all $n \geq 1$ and $i = 2, \ldots, d$, 
\beqn
(I_n: (x_i^i))
=
\begin{cases}
T^{\prime}                &  \mbox{ if } n <i \\
 I_{n-i+1}   &  \mbox{ if } n \geq i.
\end{cases}
\eeqn
\end{proposition}
\begin{proof}
If $n<i,$ then $x_i^i \in J_{i-1} \subseteq I_n$ which implies that  $(I_n:(x_i^i ))=T^\prime.$ Therefore, for the rest of the proof  we will  assume that $n \geq i.$

 As $I_n = \displaystyle \sum_{a_1 + 2 a_2 + \cdots + (d-1) a_{d-1}=n} J_1^{a_1} \cdots J_{d-1}^{a_{d-1}},$ 
by \cite[Proposition~1.14]{ene-herzog}, we only need to consider      $M_j \in J_j^{a_j}$ with $\deg(M_j) = (j+1) a_j $ and show that
${\displaystyle ((\prod_{j=1}^{d-1}M_j): (x_i^i)) \subseteq I_{n-i+1}}.$ 
Note that 
\beqn
     \left(  \left( \prod_{j=1}^{d-1}M_j\right) :(x_i^i )\right)
= \left( \f{\left( { \prod_{j=1}^{d-1}M_j}\right) }
                {gcd(  \prod_{j=1}^{d-1}M_j, x_i^i) }     \right)
=  \left( \f {\left(  { \prod_{j=1}^{i-1}  {M_j}}\right)}
     {x_i^g} \left[  {\prod_{j=i}^{d-1}M_j}\right] \right)\\
\eeqn
where $g = \min\{ i, \sum_{j=1}^{i-1} b_j\}$ and $b_j := \max \{ t | x_i^t  \mbox{ divides } M_j\}$.

If $b_j=0$ for all $j=1, \ldots, i-1$, then  $g=0$ and  ${ \displaystyle\left(\prod_{j=1}^{d-1}M_j\right): (x_i^i ) = \left(\prod_{j=1}^{d-1}M_j \right ) \subseteq I_n \subseteq I_{n-i+1}}$. 
Hence, for the rest of the proof  we will assume that $b_j \not =0$ for some $j=1, \ldots, i-1$.

{\bf Claim:} For $j=1, \ldots, i-1$, there exist integers $a_j^{\prime}$ and monomials 
$M_j^{\prime}$ such that:
\been
\item
\label{claim one}
$M_j^{\prime}\in J_j^{ a_j^{\prime} }$ for all $j=1, \ldots, i-1$.

\item
\label{claim two}
 $
{\displaystyle 
 \f{  \left( { \displaystyle\prod_{j=1}^{i-1}  {M_j}} \right) }
           {x_i^g}
= \left( \prod_{j=1}^{i-1}  M_j^{\prime}\right) N
}$, for  some monomial  $N$ in $T^{\prime}$.
 
 \item
 \label{claim three}
${\displaystyle \sum_{j=1}^{i-1} j a_j^{\prime} + \sum_{j=i}^{d-1} j a_j \geq n-i+1}$. 
\eeen

Put ${ \displaystyle k := \min\left\{ l \left| \sum_{j=l}^{i-1} b_j \leq i-1 \right.\right\}}$. 
For  $k \leq j \leq i-1$ we define $q_j$ and $r_j$ using the following algorithm:\newline
\begin{algorithm}[]
\caption{}
\begin{algorithmic}[1]
 \Ensure{Defines $q_j,r_j$ for $k \leq j \leq i-1.$}
 \Require{$r_i=0$ and $j=i-1$}
 \Statex
\While{$j \geq k$}
\If{$b_j=0$}
\State{define $q_j=0$ and $r_j=r_{j+1}$}
\Else 
\State { find integers $q_j$ and $0 \leq r_j \leq j$ such that}
\beq
\label{defn of qj}
b_j-r_{j+1}=(j+1)q_j-r_j.
\eeq
\EndIf
\State \Return{$q_j,r_j$}
\State $j\gets j-1$
\EndWhile
 \end{algorithmic}
\end{algorithm}

Define non-negative integers $q_{k-1}$ and $r_{k-1}$ as follows: Put
\beq
\label{definition of c-1}
 c&:=&g- \sum_{j=k}^{i-1} b_j.
 \eeq
  If $c=0$, then put $q_{k-1}:=0$ and $r_{k-1}:=r_k$. If $c>0$, then choose $q_{k-1}\geq 0$ and $0 \leq r_{k-1} \leq k-1$ such that 
\beq
  \label{defintion of c-2}
  c-r_{k}&=&kq_{k-1}-r_{k-1}.
 \eeq
For $j=1, \ldots, i-1$  we define $a_j^{\prime}$ as follows: 
\beq
\label{a prime}
a_{j}^{\prime}
:= \begin{cases}
a_j- q_j & \mbox{ if }  j \in\{ k-1, \ldots, i-1\} \setminus \{ r_{k-1} -1\}\\
a_j &  \mbox{ if }  j\in \{1, \ldots, k-2\} \setminus \{ r_{k-1} -1\}\\
a_{r_{k-1}-1} +1 &  \mbox{ if } j = r_{k-1}-1 \mbox{ and }  r_{k-1} \geq 2 .\\
\end{cases}
\eeq
Set $M_0=1$ and define $N_j$ for $j = k-1, \ldots, i $ as follows:
\begin{eqnarray*}
 N_j=\begin{cases}
      1 & \mbox{ if } j=i\\
      \mbox{ a monomial of degree $r_j$ that divides }\frac{M_jN_{j+1}}{x_i^{b_j}} &\mbox{ if } b_j \neq 0 \mbox{ and }  k\leq j < i\\
      N_{j+1} &\mbox{ if } b_j = 0 \mbox{ and } k\leq j < i\\
      \mbox{ a monomial of degree $r_{k-1}$ that divides } \frac{M_{k-1}N_k}{x_i^c} &\mbox{ if } j=k-1.
     \end{cases}
\end{eqnarray*}
For $j=1, \ldots,i-1$ we define $M_j^{\prime}$ as follows:
\beqn
M_j^{\prime}
:= \begin{cases}
{\displaystyle \f{M_j N_{j+1}}{x_i^{b_j} N_j} }&  \mbox{ if } j \in\{ k, \ldots, i-1\} \setminus \{ r_{k-1} -1\} \\
{\displaystyle \f{M_{k-1} N_{k}} {x_i^c N_{k-1}}} & \mbox{ if } j = k-1\\
M_j   & \mbox{ if } j\in\{1, \ldots, k-2\} \setminus \{ r_{k-1} -1\}\\
 M_{r_{k-1}-1} N_{k-1} & \mbox{ if } j= r_{k-1}-1   \mbox{ and } 2 \leq r_{k-1}.\\
\end{cases}
\eeqn
By our definition of $M_j^\prime$, $\deg(M_j^{\prime})= (j+1) a_j^{\prime}$ for all $j=1, \ldots, i-1$. Hence $M_j^{\prime} \in J_j^{ a_j^{\prime} }$. This proves (\ref{claim one}) of the Claim. 

Let  
\beqn
N=\begin{cases}
         N_{k-1} &\mbox{ if } r_{k-1} \leq 1\\
         1 &\mbox{ if } r_{k-1} > 1.
        \end{cases}
\eeqn 
Then we can express
 ${\displaystyle \prod_{j=1}^{i-1} M_j}$ as  in  (\ref{claim two}) of the Claim. 

We now   prove (\ref{claim three}) of the Claim.  To complete the proof it suffices to show that 
${ \displaystyle\sum_{j  =1}^{i-1} j a_j^{\prime} +  \sum_{j = i}^{d-1} j a_j \geq n -i + 1}$.
Put 
\beqn
\alpha(r_{k-1})
&:=&
\begin{cases}
0  &  \mbox{ if } r_{k-1}=0 \\
r_{k-1}-1 & \mbox{ if } r_{k-1} \not = 0.
\end{cases}
\eeqn
Then
\beq
\label{adding}\nonumber
&&              \sum_{j  =1}^{i-1} j a_j^{\prime} 
         +   \sum_{j = i}^{d-1} j a_j   
       \\ \nonumber
&=& n -  \sum_{j=k-1}^{i-1}[(j+1) q_j ]  
         +   \sum_{j=k-1}^{i-1}q_j  +  \alpha(r_{k-1})  
           \hspace{1.8in} \mbox{[by (\ref{a prime})]} \\ \nonumber
&=& n-  [c- r_{k} + r_{k-1}] 
-            \sum_{j=k}^{i-1} [b_j - r_{j+1}+ r_j]
+            \sum_{j=k-1}^{i-1}q_j + \alpha(r_{k-1})
              \hspace{.5in} \mbox{[by (\ref{defintion of c-2}) and (\ref{defn of qj})]}\\ \nonumber
&=& n - g + [ \alpha(r_{k-1} )- r_{k-1}]  
+           \sum_{j=k-1}^{i-1}q_j  
                 \hspace{2.2in} \mbox{[by (\ref{definition of c-1})]} .
\eeq
We claim that:\newline
 (a)  ${\displaystyle \sum_{j=k-1}^{i-1}q_j \geq 1}.$ \newline
 (b) If  $g=i$ and $r_{k-1} > 0$,  then  ${\displaystyle \sum_{j=k-1}^{i-1}q_j \geq 2}.$

Suppose  $q_j=0$ for all $j=k-1, \ldots, i-1$.  Then 
\beqn
         0 
\leq   g 
=      \sum_{j=k}^{i-1} b_j + c
=      \sum_{j=k}^{i-1} [r_{j+1} - r_j] + r_k - r_{k-1}
= -     r_{k-1}
\leq  0,
\eeqn
which implies that $g=0$. Hence $\sum_{j=k}^{i-1}b_j=0$ which leads to a contradiction on our assumption of $b_j$'s. This proves  (a) of the claim.

Now suppose  $g=i$ and $r_{k-1} > 0$. By (a),  ${\displaystyle \sum_{j=k-1}^{i-1}q_j \geq 1}.$
If ${\displaystyle \sum_{j=k-1}^{i-1}q_j=1},$ then 
$q_l=1$ for some $k-1 \leq  l\leq i-1$ and $q_j=0$ for $j\neq l.$ Hence
 \beqn
         i = g
= \sum_{j=k}^{i-1} b_j + c = (l+1) - r_{k-1} \leq i-r_{k-1} \leq i-1
\eeqn
which leads to a contradiction. 

If $g \leq i-1$ or $g=i$ and $r_{k-1}=0$, then by Claim (a), 
${\displaystyle -g +[ \alpha(r_{k-1} )- r_{k-1}]  +  \sum_{j=k-1}^{i-1}q_j \geq -i+1}$. 
If $g=i$ and $r_{k-1} \not =0$, then by Claim (b) 
${\displaystyle -g +[ \alpha(r_{k-1} )- r_{k-1}]  +  \sum_{j=k-1}^{i-1}q_j \geq -i+1}$. 
This completes the proof of (\ref{claim three}) of the Claim.
 \end{proof}

\begin{theorem}
\label{cohen macaulayness of G}
The associated graded ring $G(\F)$ is Cohen-Macaulay. 
\end{theorem}
\begin{proof} Let ${a}^{\star}$ denote the image of $a$ in $G(\F).$ Since $x_i^i \in I_{i-1} \setminus I_{i}$ it follows that  $(x_i^i)^\star \in [G(\F)]_{i-1}.$ To prove the theorem it is enough to show that 
$(x_2^2)^{\star}, \ldots ,(x_i^i)^{\star}$ is a regular  sequence in $G(\F)$ for all $2 \leq i \leq d.$  We prove by induction on $i$. If $i=2$, then by 
Proposition~\ref{symbolic ideal quotients}, ${(x_2^2)}^{\star}$ is a regular element in $G(\F)$. Now let $i>2$ and assume that 
$ {(x_2^2)}^{\star}, \ldots, {(x_{i-1}^{i-1})}^{\star}$ is a regular sequence in $G(\F)$. 
Then 
\beqn
\f{G(\F) }{({(x_2^2)}^{\star}, \ldots, {(x_{i-1}^{i-1})}^{\star})  }
&\cong& \bigoplus_{n \geq 0} \f{I_n}
                                                 { I_{n+1} + \sum_{j=2}^{i-1} {x_j^j} I_{n+1-j}} .
                                                 \eeqn
One can verify that  
\beqn
      { \displaystyle ((I_{n+i}  + \sum_{j=2}^{i-1} {x_j^j} I_{n+i-j}) : (x_i^i)) } 
 &=& {\displaystyle  (I_{n+i}    :( x_i^i) ) +   \sum_{j=2}^{i-1}( {x_j^j} I_{n+i-j}: (x_i^i))}
              \hspace{.3in}
              \mbox{\cite[Proposition~1.14]{ene-herzog}} \\
&=& {\displaystyle (I_{n+i}    :( x_i^i) ) +   \sum_{j=2}^{i-1} {x_j^j}( I_{n+i-j}: (x_i^i)) }\\
 &=&{\displaystyle I_{n+1}  + \sum_{j=2}^{i-1} {x_j^j} I_{n+1-j}}  \hspace{1.2in}
  \mbox{[by Proposition~\ref{symbolic ideal quotients}]}.
 \eeqn
 Hence  ${(x_{i}^{i}})^{\star}$ is  ${G(\F) } / {({(x_2^2)}^{\star}, \ldots, {(x_{i-1}^{i-1})}^{\star})}$- regular.   
\end{proof}

\begin{proposition}
\label{computing full length}
Let $2 \leq k \leq  d$. Then for all $n \geq 1$,
\beqn
    \ell   \left( \f{\Tprime}{(I_n + (x_2^2,  \cdots, x_{k}^{k} ))\Tprime}\right)
= \sum_{i=0}^{k-1} (-1)^i
   \left[  \sum_{1 \leq j_1 < \cdots < j_i \leq k-1} 
    \ell \left( \f{\Tprime}{(I_{n - (j_1 + \cdots + j_{i} )} )\Tprime}  \right) 
\right]   .
\eeqn
\end{proposition}
\begin{proof} The proof follows from Proposition~\ref{vanishing} and Theorem~\ref{cohen macaulayness of G}.
\end{proof}

\section{Monomial generators of $I_{n-1}$ modulo $(I_n : (x_d))$ as a $\kk$-vector space }

In this section we  first show that $(I_n : (x_d)) \subseteq I_{n-1}$. Next we describe the generators of $I_{n-1}$ modulo $(I_n : (x_d))$ (Proposition~\ref{a k basis}). This will be used to compute $\ell (T^{\prime} / I_{n} )$ and   $\ell (T^{\prime} / I_{n}  
+ (x_2^2, \ldots, x_{k+1}^{k+1}))$.

The following lemma is simple, but we state it as it is crucially used to prove Lemma~\ref{reduction step}.

  \begin{lemma}
 \label{powers of momomials}
 \begin{enumerate}
  \item
 \label{powers of momomial-one}
   Let $1 \leq j \leq d-1$ and $a \geq 1$. Then  
 \beqn
  (M_{j+1,d})^{ (j+1)a }
   = x_{j+1}^{(j+1) a -j }( M_{j+1,d})^j
 + (M_{j+2,d})^{j+1} (M_{j+1,d})^{ (j+1) (a-1)} . 
  \eeqn 
 \item
  \label{powers of momomial-two}  
 Let $1 \leq k < j \leq d-1$ and $a,b \geq 1.$ Then 
      \beqn
    (M_{k+1,d})^a (M_{j+1,d})^b =  (M_{k+1,j+1})^{a} (M_{j+1,d})^{b} +  (M_{k+1, d})^{a-1}  (M_{j+2, d})^{b+1}.
\eeqn
\end{enumerate}
  \end{lemma}
 \begin{proof} (\ref{powers of momomial-one}) The proof follows by induction on $a.$ (\ref{powers of momomial-two}) The proof follows by induction on $a + b$.  
 \end{proof}

Before we proceed we set up some notation. 
For   $1 \leq j\leq d-1$,    let $a_j \not = 0$  and ${\bf a_j}:=(a_1,\ldots,a_{j}) \in \N^{j}$. 
 We inductively define the set $S( {\bf a_j})$ as follows:
 \begin{eqnarray*}
           S({\bf a_1})
  &:=& \{x_2^{2a_1-1} \} \\
           S({\bf a_j}) 
 &:=& \begin{cases}
              \{x_{j+1}^{(j+1) a_j -j}     \}                          & \mbox{ if } \{i<j | a_i \neq 0 \} = \emptyset\\       
             \{x_{j+1}^{(j+1) a_j -j}  \} S({\bf a_k})   M_{k+1, j+1}^k &  \mbox{ if }  \{i < j | a_i \neq 0 \} \not = \emptyset \mbox{ and } k =  \max\{i<j | a_i \neq 0 \}.
                            \end{cases}
 \end{eqnarray*}
 We set  ${\bf J}^{\bf a_j}:=J_1^{a_1} \cdots J_{j}^{a_{j}}$. Let   $\wt {\bf a_j}:=a_1+2a_2+\cdots+j a_{j}$ be the weight of 
 ${\bf a_j}$.  For all   $n \in \N$ we define   $\Lambda_{j,n}:=\{{\bf a_{j}} \in \N^{j}:\wt {\bf a_j} = n, ~~a_j \neq 0 \}.$

 \begin{lemma}
 \label{reduction step}
 Let $n \geq 2$. Then 
 \been
 \item
 \label{reduction step -2}
$(I_n : (x_d)) \subseteq I_{n-1}$.
 
 \item
  \label{reduction step -1}
 For all $1 \leq j \leq d-1$ and for all   ${\bf a_j} \in \Lambda_{j,n-1}$
 \been
 \item
 \label{reduction step zero}
$   {\displaystyle S({\bf a_j}) 
   M_{j+1, d}^j \subseteq {\bf J}^{\bf a_j}
    \setminus \mprime  {\bf J}^{\bf a_j} 
} $
   where $\mprime = (x_2, \ldots, x_d)\Tprime$.  
\item
\label{reduction step one}
For all $1 \leq j \leq d-1$, 
${\displaystyle
{\bf J}^{ \bf a_j}
\subseteq  \left(S({\bf a_j}) M_{j+1,d}^j \right)  
                  + {(I_n : (x_d))}.
                 } $

\eeen

\item
\label{reduction step three}
$I_{n-1} = \sum_{j=1}^{d-1} \sum_{{\bf a_j} \in  \Lambda_{j,n-1} }\left(S({\bf a_j}) M_{j+1,d}^j \right)   + (I_n : (x_d))$.
\eeen
 \end{lemma}
\begin{proof} \eqref{reduction step -2}
By \cite[Proposition 1.14]{ene-herzog} it is enough to show that for all $j = 1 \dots d-1$ and ${\bf a_j } \in  \Lambda_{j,n},$ $({\bf J}^{\bf a_{j}}: (x_d)) \subseteq I_{n-1}.$  One can verify that 
\beqn
                        ( {\bf J}^{\bf a_{j}}: (x_d ))
&=&                (M_{2,d})^{2a_1} 
\cdots             (M_{j,d})^{ j a_{j-1}} ( M_{j +1,d})^{ (j+1) a_{j}-1}\\
&=&                (M_{2, d})^{2a_1} 
\cdots           [  (M_{j,d})^{ j a_{j-1} }    (M_{j+1 , d})^{j}]  
                       (M_{j+1 , d})^{ (j+1) a_{j}- (j+1)}\\
& \subseteq&  (M_{2, d})^{2a_1}  
\cdots              (M_{j,d})^{ j (a_{j-1}+1) } 
                        (M_{j+1, d})^{(j+1) (a_{j}-1)} 
                        \hspace{1.0in} [\mbox{as } (M_{j+1, d}) \subseteq (M_{j,d})] \\
 &\subseteq&     I_{n-1},
 \eeqn
since $a_1 + \cdots + (j-2)a_{j-2}+(j-1)( a_{j-1}+1) + j (a_{j}-1) = n-1$.
This proves (\ref{reduction step -2}). 

 \eqref{reduction step -1} Set $r({\bf a_j})= \# \{ i : 1 \leq i \leq j   \mbox{ and } a_i  \neq 0\}$.  We prove 
by induction on $r({\bf a_j})$. 

(\ref{reduction step zero}) If $r({\bf a_j})=1$, then  $S({\bf a_j}) =\{  x_{j+1}^{(j+1) a_j -j}\} .$ Hence  $   {\displaystyle S({\bf a_j}) 
   M_{j+1, d}^j}=\{  x_{j+1}^{(j+1) a_j -j}\} M_{j+1, d}^j \subseteq  J_j^{a_j} = J_j^{\bf a_j} .$ 

If $ r({\bf a_j})>1$ and $k = \max\{ i | 1 \leq i < j \mbox{ and } a_i \neq 0\}$, then 
$                     S({\bf a_j}) M_{j+1, d}^j
=               S({\bf a_k}) M_{k+1, j+1}^k
                           \left[   x_{j+1}^{(j+1) a_j - j}M _{j+1,d}^j \right] $ 
                           and by induction hypothesis,
      \beqn
   S({\bf a_k}) M_{k+1, j+1}^k     \left[   x_{j+1}^{(j+1) a_j - j}M _{j+1,d}^j \right]   
\subseteq   {\bf J}^{\bf a_k} J_j^{a_j} 
         =      {\bf J}^{\bf a_j}.
   \eeqn   
   Comparing the degree of the monomials in  $ S({\bf a_j}) M_{j+1, d}^j$   we  conclude that these monomials are not in  $\mprime  {\bf J}^{\bf a_j}$. 
                 
(\ref{reduction step one}) If $r({\bf a_j})=1$ ,then  
\beqn
         {\bf J}^{\bf a_j} 
&=& (M_{j+1,d})^{(j+1)a_j} \\
&=& x_{j+1}^{(j+1) a_j -j }(M_{j+1,d})^j
 + (M_{j+2,d})^{j+1} (M_{j+1,d})^{ (j+1) (a_j-1)}  \hspace{.2in} \mbox{[by Lemma \ref{powers of momomials}(\ref{powers of momomial-one})]}\\
 &\subseteq& S({\bf a_j})(M_{j+1,d})^j + (I_n: (x_d)) 
\eeqn
as
$ 
 x_d  (M_{j+2,d})^{j+1}  (M_{j+1,d})^{ (j+1) (a_j-1)} \subseteq J_{j+1} J_j^{a_j-1}
 $
and  $(j+1) + j(a_j-1) = ja_j + 1 = (n-1) + 1 =n$. Hence   (\ref{reduction step one}) is true for $r({\bf a_j})=1$. 
 
Now let $ r({\bf a_j})>1$ and $k = \max\{ i | 1 \leq i < j \mbox{ and } a_i \neq 0\}$. Then 
\beqn
&& {\bf J}^{\bf a_j} \\
&=&                      {\bf J}^{\bf a_k} J_j^{a_j}\\ \nonumber
&\subseteq&       \left( 
                        ( S({\bf a_k}) M_{k+1,d}^k) + (I_{n-ja_j} : (x_d)) \right)J_j^{a_j} \\
     &&               \hspace{4in} [\mbox{by induction hypothesis applied to ${\bf J}^{\bf a_k} $}]\\
 &\subseteq&         x_{j+1}^{(j+1) a_j - j }
                       ( S({\bf a_k}) M_{k+1,d}^k )(M_{j+1,d} )^j 
                       +     \left(  S({\bf a_k}) M_{k+1,d}^k \right) (I_{ja_j+1 } : (x_d)) 
                       +  (I_{n-ja_j} : (x_d)) J_j^{a_j}  \\
      &&                \hspace{4.5in} [\mbox{by the case $r=1$ applied to $J_j^{a_j} $}]  \\ 
& \subseteq &  x_{j+1}^{(j+1) a_j - j }
                       ( S({\bf a_k}) M_{k+1,d}^k) (M_{j+1,d} )^j 
          +              (I_n : (x_d)) \hspace{2.45in} \mbox{[by Lemma~\ref{reduction step}\eqref{reduction step zero}]}\\
&\subseteq&                x_{j+1}^{(j+1) a_j - j }
                       ( S({\bf a_k}) )\left[ (M_{k+1,j+1})^{k} (M_{j+1,d})^{j} 
               +      (M_{k+1, d})^{k-1}  (M_{j+2, d})^{j+1}\right]
                 +    (I_n : (x_d)) 
                       \hspace{.3in} \mbox{[by Lemma~\ref{powers of momomials}(\ref{powers of momomial-two})]}\\
 &=&              \left( S({\bf a_j}) (M_{j+1,d})^{j}\right)
  +                   ( I_n: (x_d))
                      \eeqn
                      as
                      \beqn
  &&              x_d ( {\displaystyle    x_{j+1}^{(j+1) a_j - j }})
              ( S({\bf a_k})   ) (M_{k+1, d})^{k-1} (M_{j+2, d})^{j+1}\\ \nonumber
&\subseteq& \left[ \left( S({\bf a_k}) \right) (x_{j+1}  (M_{k+1, d})^{k-1})\right]
             x_d( M_{j+2, d})^{j+1} (x_{j+1}^{(j+1)(a_j -1) } )\\ \nonumber
&\subseteq& {\bf J}^{\bf a_k} J_{j+1} J_j^{a_j-1} \hspace{3.5in} \mbox{[by Lemma~\ref{reduction step}\eqref{reduction step zero}]}\\
&\subseteq & I_n. 
\eeqn
This proves \eqref{reduction step one} for all $r({\bf a_j}) \geq 1$.

 \eqref{reduction step three} The proof follows from  \eqref{reduction step -2} and  \eqref{reduction step -1}.
\end{proof}
     
 \begin{proposition}
 \label{a k basis}
The set
$
 \{M + (I_n : (x_d)) | M \in \{\Vx_{j=1}^{d-1} {\Vx}_{\bf a_j \in  \Lambda_{j,n-1}}\{ S({\bf a_j})  M_{j+1, d}^j  \}\}
$ generates $ {\displaystyle {\f{I_{n-1}}{ (I_n: (x_d))}} }$ as a $\kk$-vector space.
\end{proposition}
\begin{proof}
Let  $M$ be  a monomial in $S({\bf a_j})  M_{j+1, d}^j.$ By Lemma \ref{reduction step}\eqref{reduction step zero},  $M \in {\bf J}^{\bf a_j}.$ Thus $x_d x_i M \in ({\mprime})^2 J_1^{a_1} \cdots J_{j}^{a_{j} }= J_1^{a_1+1} \cdots J_{j}^{a_{j} }\subseteq I_n. $ This implies that $x_i M \in (I_n : (x_d))$  for all $i=2, \ldots, d$. Hence from 
Lemma~\ref{reduction step}\eqref{reduction step three}, the monomials in $S({\bf a_j})  M_{j+1, d}^j $ generate $I_{n-1}$ 
modulo $(I_n : (x_d))$ as a $\kk$-vector space. 
\end{proof}

From Proposition~\ref{a k basis},  giving  an upper bound for the length of the vector space  $ {\displaystyle {\f{I_{n-1}}{ (I_n: (x_d))}} }$  involves counting monomials  and hence it  is combinatorial in nature. Hence we prove some preliminary results before we arrive at the main result of this section.   We state  the well known  Vandermonde's identity which will be needed in our proofs.  
\begin{lemma}
\label{vandermonde}
\rm[Vandermonde's identity]
Let $n,r, s \in \mathbb \N $.  Then 
\beqn
\sum_{i \geq 0} {n \choose i}{s \choose  r-i} = {n + s \choose r}
\eeqn
\end{lemma}

The next lemma is  the main step in proving our main result. 
         
\begin{lemma}
\label{number of sj}
 Fix $1 \leq j \leq d-1$ and $n >1$. Then
${\displaystyle
   \sum_{{\bf a_j } \in \Lambda_{j,n-1}} \# S({\bf a_j})
= \binom{n-2}{j-1}.}
$
\end{lemma}
\begin{proof}
 We prove by induction on $j.$ If $j=1$ then  $S({\bf a_1})= \{x_2^{2a_1 -1}\}$,  and hence  the assertion is true for $j=1.$ 
 
 Now let $j >1.$    Then
 \begin{eqnarray}
 \label{ computation of Saj}\nonumber
&&     \sum_{{\bf a_j } \in \Lambda_{j,n-1}} \#S({\bf a_j})\\
 &=&  \begin{cases}
        {\displaystyle  \sum_{a_j=1}^{\lfloor \frac{n-2}{j} \rfloor} 
          \sum_{i=1}^{j-1} \left[ \sum_{{\bf a_i} \in \Lambda_{i,n-1-ja_j}} 
          \#S({\bf a_i}) \right] \#M_{i+1,j+1}^i  }
  &       \mbox{ if }j\not |  (n-1 )\\ 
          \#S(0,\ldots,0,\frac{n-1}{j}) 
+       {\displaystyle \sum_{a_j=1}^{\lfloor \frac{n-2}{j} \rfloor}
          \sum_{i=1}^{j-1} \left[
          \sum_{{\bf a_i} \in
           \Lambda_{i,n-1-ja_j}} 
           \#S({\bf a_i}) \right]
           \#M_{i+1,j+1}^i}
   &     \mbox{ if } j | (n-1)
\end{cases}.
\eeq
Define 
$ {\displaystyle
\alpha_{j,n}
:=
\begin{cases}
0  &  \mbox{ if } j \not | (n-1 )\\
1 & \mbox{ if } j| (n-1)
\end{cases}
}
.
$
Then   \eqref{ computation of Saj} can be written as
\beqn
&&     \sum_{{\bf a_j } \in \Lambda_{j,n-1}} \#S({\bf a_j})\\
  &=& \alpha_{j,n} + \sum_{a_j=1}^{\lfloor \frac{n-2}{j} \rfloor} 
        \sum_{i=1}^{j-1}\left[ \sum_{{\bf a_i} \in \Lambda_{i,n-1-ja_j}} 
        \#S({\bf a_i}) \right]\#M _{i+1,j+1}^i \\
  &=&   \alpha_{j,n} 
  +        \sum_{a_j=1}^{\lfloor \frac{n-2}{j} \rfloor} \sum_{i=1}^{j-1} 
               \binom{n-ja_j-2}{i-1}
               \binom{j}{j-i} \hspace{1.95in} \mbox{ [by induction hypothesis] }\\
  &=&  \alpha_{j,n} 
  +       \sum_{a_j=1}^{\lfloor \frac{n-2}{j} \rfloor} 
           \left[  \left[ \sum_{i=0}^{j-1}  
           \binom{n-ja_j-2}{i}  \binom{j}{j-i-1}\right]
 -         \binom{n - j a_j -2}{j-1} \right] \hspace{.2in} \mbox{[replacing  $i-1$ by $i$]}\\
 &=&   \alpha_{j,n} 
 +        \sum_{a_j=1}^{\lfloor \frac{n-2}{j} \rfloor} 
           \left[   
           \binom{n - j (a_j-1) -2}{j-1} -  \binom{n - j a_j -2}{j-1}
           \right]           \hspace{1.1in} \mbox{[by Lemma~\ref{vandermonde}]}\\
  &=& \alpha_{j,n} +  \binom{n-2}{j-1}  - \alpha_{j,n}\\
   &=& \binom{n-2}{j-1}.
 \end{eqnarray*}
\end{proof}

 We are now ready to prove the main result in this section.        
\begin{proposition}
\label{length of last term}
Let $d,n \geq 2$. Then
\beqn
       \ell \left(  \f{I_{n-1}} {(I_n : (x_d))}\right)
\leq                         {n + d-3 \choose d-2}.
\eeqn
\end{proposition}
\begin{proof} By  Proposition~\ref{a k basis} we get
\beqn
           \ell \left(  \f{I_{n-1} } 
                           {(I_n : (x_d))} \right)
&\leq&  \sum_{j=1}^{d-1}  
         \left[ \sum_{ {\bf a_j} \in \Lambda_{j, n-1} }
          \# S({\bf a_j})         \right]   \# M_{j+1, d}^j  \\
& =&  \sum_{j=1}^{d-1}  {n- 2 \choose j-1} {d-1 \choose d-j-1} 
          \hspace{.6in} \mbox{[by Lemma~\ref{number of sj}]}.\\
&=&   \sum_{i = 0}^{d-2}{n-2 \choose  i} { d -1  \choose d-i-2} 
           \hspace{.6in} \mbox{[put  $i=j-1$]}\\
&=&     {n + d-3 \choose d-2} \hspace{1.36in}  
          \mbox{[by  Lemma~\ref{vandermonde}]}.
               \eeqn\end{proof}

\section{Cohen-Macaulayness of $R/  (\p^{(n)}    + (\ff_k))$ }
In \cite[Proposition~7.6]{goto}  Goto showed that $R/  (\p^{(n)}    + (\ff_{d-1}) )$  is Cohen-Macaulay for $d=3,4$ and $n \leq {d-1 \choose 2}$. This was done by explicitly describing $\p^{(n)}$ for $d\leq 4$ and $n \leq {d-1 \choose 2}$. Using the techniques developed in this paper, we generalise Goto's result  for all $d \geq 2$ and $n \geq 1$. A lower bound for $\ell (R/ (\p^{(n)}  + (\ff_k , x_1) ))$ was  given using the multiplicity formula (Theorem~\ref{computation of multiplicity}). In this section, we show that the inequality in 
Theorem~\ref{computation of multiplicity} is indeed an equality (Theorem~\ref{bound on length}). 
This implies  that  for all $n \geq 1$ and $1 \leq  k \leq d-1$,  the rings  $R/  (\p^{(n)}    + (\ff_k) )$  are Cohen-Macaulay.
As a consequence, we  describe $\p^{(n)}$ for all $d \geq 2$ and all  $n \geq 1$. In particular  we prove that 
 $\p^{(n)} = {\mathcal I}_nR$ and
$LI(\p^{(n)} \Tprime) = I_n\Tprime$ for all $d \geq 2$ and $n \geq 1$.  

We first give an upper bound on $\ell (T^{\prime}/ I_n)$. This is crucial to prove an interesting result which shows that the  the equality of the  lengths of the various modules (over different rings) in Theorem~\ref{main theorem 1}.

\begin{proposition}
\label{main theorem}
Let $d \geq 2$. Then for all $n \geq 1$, 
\beqn
      \ell \left( \f{\Tprime}{I_n} \right)
\leq     d{n+d-2 \choose d-1}.
\eeqn
\end{proposition}
\begin{proof} We prove by double induction on $n$ and $d$. If $n =1$, then 
\beqn
\ell \left( \f{\Tprime}{I_1} \right)
= \ell \left( \f{k[x_2, \ldots, x_d]}{(x_2, \ldots, x_d)^2}\right)
 = d.
 \eeqn
If  $d=2$, then
\beqn
\ell \left( \f{\Tprime}{I_n} \right)
 = \ell \left( \f{k[x_2]}{(x_2)^{2n}}\right)
= 2n.
\eeqn
Now let $n > 1$ and $d>2$.
From the  exact sequence
\beqn
        0 
\lrar \f{\Tprime}{(I_n: (x_d))} 
\stackrel{.x_d}{\lrar} \f{\Tprime}{I_n}
\lrar \f{\Tprime}{I_n + (x_d)}
\lrar 0  
\eeqn
we get
\beqn
&&        \ell \left( \f{\Tprime}{I_n}\right)\\
&=&      \ell \left( \f{\Tprime}{I_n + (x_d)}\right)
+            \ell \left( \f{\Tprime}{(I_n : (x_d))}\right)\\
&=&       \ell \left( \f{\Tprime}{I_n + (x_d)}\right)
+            \ell \left( \f{\Tprime}{I_{n-1}}\right)
+            \ell \left( \f{I_{n-1}}{(I_n : (x_d))}\right) 
              \hspace{0.95in} \mbox{[Lemma~\ref{reduction step}(\ref{reduction step -2})]}\\
&\leq&        (d-1) {n + d-3 \choose d-2}
+             d {n-1 + d-2 \choose d-1}
+            {n + d-3 \choose d-2} 
             \hspace{.1in} 
             \mbox{[by induction hypothesis and Proposition~\ref{length of last term}]}\\
&=&       d{n + d-3 \choose d-2 } 
+           d {n + d-3 \choose d-1}\\
&= &      d {n + d-2 \choose d-1}.
\eeqn
\end{proof}

\begin{theorem}
\label{main theorem 1}
Let $d \geq 2$. Then for all $n \geq 1$, 
\beqn
    e   \left(x_1; \f{R}{\p^{(n)}}\right)
= \ell \left(  \f{R}{\p^{(n)} + (x_1)}\right)
&=&  \ell_R \left(\f{R}
                        { ({\mathcal I}_{n}, x_1) R}\right)
=    \ell_{\Tprime} \left(  \f{\Tprime} 
                                            { {\mathcal I}_{n}\Tprime }  
                                            \right)\\
&=&  \ell_{\Tprime} \left(  \f{\Tprime} 
                                     { LI( {\mathcal I}_{n} )\Tprime }   \right)  
= \ell \left( \f{\Tprime}{I_n} \right)
= d{n+d-2 \choose d-1}.
\eeqn
\end{theorem}
\begin{proof} 
From  Proposition~\ref{description of In}(\ref{description  of In one})  ${\mathcal I}_n  R\subseteq   \p^{(n)}$.  
Since $R/ \p^{(n)}$  is Cohen-Macaulay, 
\beq
\label{equality of all terms 1}
 e\left(x_1;\f{R}
                         {\p^{(n)}}\right)
=    \ell_R \left(\f{R}
                                  {\p^{(n)}+(x_1)}\right) 
                           \leq         \ell_R \left(\f{R}
                                  { ({\mathcal I}_{n}, x_1) R}\right).            
                                  \eeq
  By 
 Proposition~\ref{description of In}(\ref{description  of In three}), for any prime $\q \not = \m $,  $(({\mathcal I}_n , x_1)T)_{\q} = T$. This implies that 
${\displaystyle \supp_T \left(  \f{T  }{ ( {\mathcal I}_{n}, x_1)T} \right) =\{\m\}} $.
  Hence  we get  
  
\beq
\label{equality of all terms 2}\nonumber
            \ell_R \left(\f{R}
                                  { ({\mathcal I}_{n}, x_1) R}\right)    \nonumber                   
&=&       \ell_{\Tprime} \left(  \f{\Tprime} 
                                            { {\mathcal I}_{n} \Tprime }  
                                            \right)              \hspace{1.35in} \mbox{[Lemma~\ref{comparing lengths}]}     \\    \nonumber
  &=&  \ell_{\Tprime} \left(  \f{\Tprime} 
                                             { LI( {\mathcal I}_{n} )\Tprime } 
                                           \right)  
 \hspace{1in}         
               \mbox{   \cite[Proposition~2.1]{bayer-stillmann}}   \\  \nonumber
 &\leq&              \ell_{\Tprime} 
              \left(\f{\Tprime}
                       {I_n}\right)      \hspace{1.5in} 
               \mbox{  [Proposition~\ref{description of In in S}]}  \\ \nonumber    
& \leq &    d \binom{n+d-2}{d-1}   
  \hspace{1.2in} 
               \mbox{  [Proposition~\ref{main theorem}]} \\     \nonumber    
 &=& e \left(x_1;\f{R}{\p}\right)
              \ell_{R_{\p}}
              \left(\f{R_{\p}}{\p^{n}R_{\p}}\right) \\ \nonumber
&=& e \left(x_1;\f{R}
                     {\p}\right)
             \ell_{R_{\p}}
             \left(\f{R_{\p}}
                       {\p^{(n)}R_{\p}}\right)     
                        \hspace{.5in} [\mbox{since } \p^{(n)}R_{\p}=\p^{n}R_{\p}]\\          
&=&    e\left(x_1;\f{R}
                         {\p^{(n)}}\right)
                \hspace{1.3in} \mbox{[by \cite[Theorem 14.7]{matsumura}]}.    
\eeq
Thus equality holds in (\ref{equality of all terms 1})  and  (\ref{equality of all terms 2}) which proves the theorem.
\end{proof}

\begin{theorem}
\label{bound on length}
Let $d \geq 2$ and $1 \leq k \leq d-1$. Let $\ff_k$ be as in \eqref{definition of f_i}. Then for all $n \geq 1$, 
\beqn
 &&   e   \left(x_1; \f{R}{\p^{(n)}   + (\ff_k)}\right)
= \ell_R \left(  \f{R}{\p^{(n)} + (x_1, \ff_k)}\right)
=    \ell_{\Tprime} \left(  \f{\Tprime} 
                                            { ({\mathcal I}_{n}  + \ff_k)\Tprime }  
                                            \right)\\
&&=  \ell_{\Tprime} \left(  \f{\Tprime} 
                                     { LI( ( {\mathcal I}_{n} + \ff_k )\Tprime  )}   \right)  
= \ell \left( \f{\Tprime}{I_n + (x_2^2, \ldots, x_{k+1}^{k+1})  } \right)\\
&&= d \sum_{i=0}^{k} (-1)^i
           \left[  \sum_{1 \leq j_1 < \cdots < j_i \leq k} 
           \binom{n-(j_1+\cdots+j_i) + d-2}{d-1} 
\right].
\eeqn
In particular, $
 {\displaystyle {R}/
{ (\p^{(n)}+(\ff_k )} )}
$
is Cohen-Macaulay. 
\end{theorem}
\begin{proof}
From  Proposition~\ref{description of In}(\ref{description  of In one})  $({\mathcal I}_n  ,x_1,\ff_k ) R\subseteq  
( \p^{(n)} ,x_1,\ff_k ) R$.  
Hence 
 \beq
\label{multiplicity inequality 1}\nonumber
                 e \left(x_1; \f{R}
                            {\p^{(n)} + (\ff_k)} \right)  
&\leq &  \ell_R \left( \f{R}
                               {\p^{(n)}+(x_1,\ff_k )}\right) \hspace{.5in} \mbox{\cite[Theorem~14.10]{matsumura}}\\ 
&\leq&   \ell_R \left(  \f{R} 
                                 { ({\mathcal I}_{n}, x_1, \ff_k)R }\right) .
\eeq
Since $({\mathcal I}_n, x_1) T \subseteq ({\mathcal I}_n, \ff_k, x_1) T$  by Proposition~\ref{description of In}(\ref{description  of In three}), for any prime $\q \not = \m $,  $(({\mathcal I}_n , \ff_k, x_1)T)_{\q} = T$. This implies that 
${\displaystyle \supp_T \left(  \f{T} { ({\mathcal I}_{n},  \ff_k, x_1)T } \right) =\{\m\} }$.   Hence  we get 
\beq   
\label{multiplicity inequality 2}
\nonumber     
&&   \ell_R \left(  \f{R} 
                                 { ({\mathcal I}_{n},x_1, \ff_k)R }\right)\\\nonumber
& = &      \ell_R \left(  \f{T} 
                                  {( {\mathcal I}_{n} ,x_1, \ff_k) T} \otimes_T  R \right)   
                                  \\ \nonumber
&=&        \ell_T \left(  \f{T} 
                                  {( {\mathcal I}_{n}, x_1, \ff_k) T}  \right)    
               \hspace{3.68in} \mbox{[Lemma~\ref{comparing lengths}]}   
              \\ \nonumber
&=&        \ell_{\Tprime} \left(  \f{\Tprime} 
                                              { ({\mathcal I}_{n} , \ff_k) \Tprime }  \right)  
               \\ \nonumber
&=&       \ell_{\Tprime} \left(  \f{\Tprime} 
                                              { LI(( {\mathcal I}_{n},  \ff_k)\Tprime )}  \right)   
              \hspace{3.05in}         
             \mbox{   \cite[Proposition~2.1]{bayer-stillmann}}
             \\ \nonumber
&\leq&  \ell_{\Tprime} \left(  \f{\Tprime} 
                                              {  { I}_{n} + ( x_2^2, \ldots, x_{k+1}^{k+1} )}  \right) 
                 \hspace{2.1in} \mbox{  [Propositions~\ref{description of In in S} and 
                 \ref{proposition ji}(\ref{proposition ji one})] }      
                 \\ \nonumber
& =&     \sum_{i=0}^{k} (-1)^i
            \left[  \sum_{1 \leq j_1 < \cdots < j_i \leq k} 
             \ell \left( \f{\Tprime}
                             {(I_{n - (j_1 + \cdots + j_{i} )} )\Tprime}  \right) 
             \right] 
               \hspace{1.65in} \mbox{  [Proposition~\ref{computing full length}]}   \\ \nonumber    
&=& d \sum_{i=0}^{k} (-1)^i
           \left[  \sum_{1 \leq j_1 < \cdots < j_i \leq k} 
           \binom{n-(j_1+\cdots+j_i)+d-2}{d-1} \right] 
         \hspace{1.3in}  \mbox{[Theorem \ref{main theorem 1}]}  \\ \nonumber
&=&d  \sum_{i=0}^{k} (-1)^{i} 
          \left[ \sum_{1 \leq j_1 < \cdots < j_i \leq k}
          \ell \left( \f{R_{\p}}{\p^{n - [{j_1} +  \cdots +{j_i} ]} R_{\p}}
          \right)\right]\\
&=& e \left(x_1; \f{R}
                            {\p^{(n)} + (\ff_k)} \right)  
                            \hspace*{1.7in} \mbox{[\cite[Proposition 5.3(3)]{goto} and Corollary \ref{multiplicity}\eqref{multiplicity-one}]}.                                  
  \eeq 
  Hence equality holds in \eqref{multiplicity inequality 1} and \eqref{multiplicity inequality 2} which proves the theorem.
\end{proof}

We end this section by explicitly describing the generators of $\p^{(n)}$ for all $n \geq 1$. We also describe the leading ideal $LI(\p^{(n)})T^{\prime}$.
\begin{theorem}
\label{symbolic power}
\been
\item
\label{symbolic power one}
For all $n \geq 1$, $\p^{(n)} = {\mathcal{I}}_n R$.

\item
\label{symbolic power two}
For all $n \geq d$, 
${\displaystyle  \p^{(n)}
= \sum_{a_1 + 2a_2+ \cdots + (d-1) a_{d-1}=n} 
  \p^{a_1} (\p^{(2)})^{a_2} \cdots (\p^{(d-1)})^{a_{d-1}}}$.

\item
\label{symbolic power three}
For all $
n \geq 1$, $LI( \p^{(n)} T^{\prime}) = I_n = LI( {\mathcal I}_n \Tprime)$. 
\eeen
\end{theorem}
\begin{proof} (1)
By Theorem~\ref{main theorem 1}  we get 
\beqn
\ell \left( \f{R} {\p^{(n) }  + (x_1)}\right)
= \ell \left( \f{R} { {\mathcal I}_n R + (x_1)}\right).
\eeqn
This implies that $\p^{(n)} =  {\mathcal I}_n R + x_1 (  \p^{(n)}  :(x_1))$. As $x_1$ is a nonzerodivisor on $R/ \p^{(n)}$, 
 $(  \p^{(n)}  :(x_1)) = \p^{(n)}$. By Nakayama's lemma, $\p^{(n)} =  {\mathcal I}_nR$. 
 
 (2)  For all $n \geq d$, 
  \beqn
  \p^{(n) } 
  &=& {\mathcal I}_n R\\
  &=& \sum_{a_1 + 2a_2+ \cdots + (d-1) a_{d-1}=n} 
           {\mathcal J}_1^{a_1}{\mathcal J}_2^{a_2} \cdots {\mathcal J}_{d-1}^{a_{d-1}} R
           \\
 &\subseteq &     \sum_{a_1 + 2a_2+ \cdots + (d-1) a_{d-1}=n} 
         \p^{a_1} (\p^{(2)})^{a_2} \cdots (\p^{(d-1)})^{a_{d-1}}
        \hspace{.5in} \mbox{[by Proposition~\ref{description of In}(\ref{description  of In one})]}\\
         &\subseteq& \p^{(n)}. 
  \eeqn
 Hence equality holds. 
 
 (3)  The proof follows from Proposition~\ref{description of In in S} and Theorem~\ref{main theorem 1}. 
\end{proof}

\section{Applications}

\subsection{Cohen-Macaulayness  and Gorensteinness of symbolic blowup algebras} \hfill\\
 In \cite{goto-nishida-shimoda}, Goto et al.  studied the Gorenstein property of the  symbolic Rees algebra. If $d=3$, then $\height(\p)= 2$ and hence,  if  $\R_s({\p})$  is Cohen-Macaulay, then  it is also Gorenstein (\cite[Corollary~3.4]{simis-trung}).   
 From \cite[Theorem~6.7(4)]{goto} and Theorem~\ref{bound on length}, it follows that $G_s({\p})$ is Cohen-Macaulay. In this paper we give an alternate argument for $G_s({\p})$ to be Cohen-Macaulay.
In fact,  we 
 show that  $G_s(\p):= \oplus_{n \geq 0} \p^{(n)} / \p^{(n+1)}$ is  Gorenstein for all $d \geq 2$ (Theorem~\ref{cm-ass-gr}). 
We also prove that  $\R_s({\p})$ is Cohen-Macaulay for all $d \geq 2$  (Theorem~\ref{cm-rees}(\ref{cm-rees-one})).
Moreover, 
 $\R_s({\p})$ is Gorenstein if and only if $d=3$ (Theorem~\ref{cm-rees}(\ref{cm-rees-two})). 

Put $f_0 = x_1$. Let $f_i$'s be as in \eqref{definition of f_i} and  let $f_{i}^{\star}$   denotes the image of $f_i$ in $\p^{(i)}  / \p^{(i+1)} $. In \cite[Proposition~5.3]{goto}, Goto showed that $\ff_{d-1}$ is a homogenous system of parameters in $G_s({\p})$. In Theorem~\ref{cm-ass-gr}, we show that  $f_0^{\star},  \ff_{d-1}^{\star}$ is a regular sequence in $G_s({\p})$.

\begin{theorem}
\label{cm-ass-gr}
Let $d\geq 2$. Then 
\been
\item
For all $d \geq 2$, $f_0^{\star},  \ff_{d-1}^{\star}$ is a regular sequence in $G_s({\p})$.

\item
$G_{s}(\p)$ is Gorenstein.
\eeen
 \end{theorem}
 \begin{proof} We first show   that $G_{s}(\p)$ is Cohen-Macaulay. 
   By induction on $k$,  we prove that  $f_0^{\star},  \ff_{k}^{\star}$  is a regular sequence in $G_{s}(\p)$ for all $k=0, \ldots, d-1$. 
Let $k=0$. Then  as $x_1$ is a nonzerodivisor on $R/ \p^{(n)}$ for all $n$, we conclude that $f_0^{\star}$ is a nonzerodivisor in $G_{s}(\p)$. 
Now let $k \geq 1$ and assume that  $f_0^{\star}, \ff_{k-1}^{\star}$ is a regular sequence in $G_{s}(\p)$.  Then 
\beqn
  \f{G_s(\p) }{ ({f_0}^{\star}, \ff_{k-1}^{\star})  }
&\cong& \bigoplus_{n \geq 0} \f{\p^{(n)}}
                                                 { \p^{(n+1)} + \sum_{j=0}^{k-1} {f_j} \p^{(n-j)}}
                                                                                                  \cong     \bigoplus_{n \geq 0}     \f{ \p^{(n)} + (f_0, \ff_{k-1})} 
         { \p^{(n+1)}  + (f_0, \ff_{k-1})}.
         \eeqn
Hence,  to show that  $f_k^{\star}$ is a nonzerodivisor on 
${\displaystyle \f{G_{s}(\p)}{({f_0}^{\star}, \ff_{k-1}^{\star})}}$ it is enough to show that 
$((\p^{(n+1)}, f_0, \ff_{k-1} )  : (f_k)) = (\p^{(n+ 1-k)}, f_0, \ff_{k-1} )$ for all $n \geq k$. 
 Since 
  \beqn
 &&         \ell \left( \f{R}
                         {((\p^{(n+1)}, f_0, \ff_{k-1} )  : (f_k))} \right) \\
 &=&  \ell \left( \f{R}{(\p^{(n+1)}, f_0, \ff_{k-1} )} \right) 
 -        \ell \left( \f{R}{(\p^{(n+1)}, f_0, \ff_{k} )} \right)  
          \\
 &=&  \ell \left( \f{T^{\prime}}{ I_{n+1} + ( x_2^2, \ldots, x_{k}^{k})}\right)
 -        \ell \left( \f{T^{\prime}}{ I_{n+1} + ( x_2^2, \ldots, x_{k+1}^{k+1})}\right) 
 \hspace{.3in} \mbox{[Theorem~\ref{bound on length}]}\\
 &=&  \ell \left( \f{T^{\prime}}{  (I_{n+1} + ( x_2^2, \ldots, x_{k}^{k})) : (x_{k+1}^{k+1} )  }\right)\\
 &=&  \ell \left( \f{T^{\prime}}{ I_{n+1-k} + ( x_2^2, \ldots, x_{k}^{k})}\right) 
             \hspace{2.1in} \mbox{[Proposition~\ref{symbolic ideal quotients} and \cite[Proposition~1.14]{ene-herzog}]}\\
 &=&  \ell \left( \f{R}
                         {(\p^{(n+1-k)}, f_0, \ff_{k-1} )} \right) \hspace{2.3in} \mbox{[Theorem~\ref{bound on length}]},
 \eeqn
 we get 
 $( (\p^{(n+1)}, f_0, \ff_{k-1} )  : (f_k)) = (\p^{(n+1-k)}, f_0, \ff_{k-1} )$. This implies that  $f_k$ is a nonzerodivisor in
  $G_{s}({\p})/ ({f_0}^{\star}, \ff_{k-1}^{\star}) $.
Hence  $G_{s}({\p})$ is Cohen-Macaulay. 
 
As $G(\p R_{\p})$ is a polynomial  ring, it is Gorenstein. Hence  by
 Theorem~\ref{bound on length} and \cite[Corollary~5.8]{goto} $G_{s}({\p})$ is Gorenstein. 
 \end{proof}
 
 \begin{theorem} 
 \label{cm-rees}
Let $d \geq 2$. Then
\been
\item
 \label{cm-rees-zero}
  $\R_s(\p)= R[\p t, {\mathcal J_2}t^2, \ldots, {\mathcal J}_{d-1}t^{d-1}]$.
\item
 \label{cm-rees-one}
 $\R_s(\p)$ is Cohen-Macaulay.

\item
 \label{cm-rees-two}
$\R_s(\p)$ is Gorenstein if and only if $d=3$.
\eeen
\end{theorem}
\begin{proof}
(\ref{cm-rees-zero}) The proof  follows from  Theorem~\ref{symbolic power}(\ref{symbolic power two}).

(\ref{cm-rees-one})
By \cite[Theorem~6.7]{goto}, it suffices to show that ${\displaystyle \f{R}{\p^{(n)}+(\ff_{d-1})}}$ is Cohen-Macaulay for $1 \leq n \leq \binom{d-1}{2}$. This holds true by Theorem~\ref{bound on length}.

(\ref{cm-rees-two})
By \cite[Lemma~6.1]{goto}, the a-invariant of $(G_{s}(\p))$,  $a(G_{s}(\p))= -(d-1)$. By \cite[Theorem~6.6]{goto}
 and Theorem~\ref{cm-ass-gr},  $\R_s(\p)$ is Gorenstein if and only if $d=3$.
\end{proof}

\subsection{Computation of resurgence} \hfill\\

In \cite{BH} C. Bocci and B. Harbourne defined the {\it resurgence} of an ideal $I$ in $R$ as 
 $$
 {\displaystyle 
 \rho(I):=\sup\left\{\frac{n}{r}:  I^{(n)}  \nsubseteq I^r \right\}.}
 $$
 We can also compute the resurgence in the following way:
For any ideal $I \subseteq R$ let  $\rho_n(I):=\min\{r: I^{(n)} \nsubseteq I^r\}.$ Then
 $$
 {\displaystyle 
 \rho(I):=\sup\left\{\frac{n}{\rho_n(I)}:  n \geq 1\right\}.}
 $$

In this subsection we explicitly describe  the resurgence of $\p=I_{{\mathcal C}(n_1, n_2, n_3)}.$

From \eqref{matrix of p} we have 
$
X  = {\displaystyle     \left(\begin{array}{ccc} x_1 & x_2 & x_3 \\x_2 & x_3 & x_1^{m+1} \\x_3 & x_1^{m+1} & x_1^m x_2\end{array}\right)}
.$
Put 
\beq
\label{def of delta}
\Delta_1 :=  \det( X_{2, (2,3)}), \hspace{.2in}
\Delta_2 :=  \det( X_{2, (1,3)}) \hspace{.2in} \mbox{ and } \Delta_3 :=  \det( X_{2, (1, 2)}).
\eeq
 Let $f_2$ be as in \eqref{definition of f_i}. 
%
\begin{lemma}
\label{f_2}
With the above notation:
\been
\item 
\label{x_if_2}
For all  $i=1,2,3$, $x_i f_2 \in \p^2$.
\item
\label{f_2^2} $f_2^2 \in \p^3$.
\eeen
\end{lemma}
\begin{proof} \eqref{x_if_2}
One can verify that 
\beqn
x_1 f_2 &=& -\Delta_2^2 + \Delta_1 \Delta_3 \\
x_2 f_2 &=& - x_1^m \Delta_3^2 -  \Delta_1\Delta_2 \\
x_3 f_2 &= &  - \Delta_1^2 - x_1^m \Delta_2 \Delta_3.
\eeqn
As $\Delta_j \in \p$ for all $j=1,2,3$,  we get $x_i f_2 \in \p^2$  for all $i=1,2,3$. 

\eqref{f_2^2} We have
 \beqn
   f_2^2& =& (x_3 \Delta_1 - x_1^{m+1} \Delta_2+x_1^mx_2 \Delta_3)f_2\\
   & =& \Delta_1(x_3 f_2) - \Delta_2  (x_1^{m+1}f_2) + \Delta_3 (x_1^mx_2f_2) \\
   & \in & \p \p^2     \hspace{3.2in} \mbox{[from \eqref{x_if_2}]}  \\
   &=&  \p^3.
  \eeqn
\end{proof}

\begin{proposition}
\label{Prop:Contain}
 Let  $k \geq 0$. Then
 \beqn
 \rho_n(\p) 
 = \begin{cases}
 3k +  1 & \mbox{ if } n = 4k\\
  3k +  2 & \mbox{ if } n = 4k+1\\
   3k +  2 & \mbox{ if } n = 4k+2\\
    3k +  3 & \mbox{ if } n = 4k+3\\
 \end{cases}
 \eeqn
\end{proposition}
\begin{proof} 
From Theorem~\ref{symbolic power}\eqref{symbolic power one} and  Theorem~\ref{symbolic power}\eqref{symbolic power two}  we get  
\begin{eqnarray}
\label{Eqn:f_2} \p^{(2)} = \p^2 + (f_2), \hspace{.2in} 
\p^{(2n)}=(\p^{(2)})^n \hspace{.2in} \mbox{and}  \hspace{.2in}
\p^{(2n+1)} =\p \p^{(2n)} . 
\end{eqnarray}
From \eqref{Eqn:f_2} and  Lemma~\ref{f_2} we get
\beqn
 \p^{(4k)}
 &= & (\p^{(4)})^k 
=  ((\p^2+(f_2))^2)^k
 = (\p^4 + f_2 \p^2 + (f_2)^2)^k
 \subseteq (\p^{3})^k = \p^{3k}\\
  \p^{(4k+1)} &=& \p \p^{(4k)} 
 \subseteq   \p \p^{3k} = \p^{3k+1} \\
 \p^{(4k+2)} &=& \p^{(2)}\p^{(4k)}  
 \subseteq   \p \p^{3k} =  \p^{3k+1} \\
\p^{(4k+3)} &=& \p \p^{(2)} \p^{(4k)}
 \subseteq \p^2 \p^{3k}  =   \p^{3k+2}.
\eeqn

As $f_2 \equiv x_3^3(\mod~x_1)$, $\Delta_3 \equiv x_3^2(\mod~x_1)$ and $\p \equiv  (x_2, x_3)^2(\mod~x_1)$
\beqn
f_2^{2k}  \equiv  x_3^{6k}&\in& \p^{(4k)} \setminus \p^{3k+1}  (\mod x_1)\\
\Delta_1 f_2^{2k}  \equiv  x_3^{6k+2} &\in& \p^{(4k+1)} \setminus \p^{3k+2}  (\mod x_1)\\
 f_2^{2k+1}  \equiv  x_3^{6k+3} &\in& \p^{(4k+2)} \setminus \p^{3k+2}  (\mod x_1)\\
 \Delta_1 f_2^{2k+1}  \equiv  x_3^{6k+5} &\in& \p^{(4k+3)} \setminus \p^{3k+3}  (\mod x_1).
\eeqn
This completes the proof.
\end{proof}

\begin{theorem}
\label{resurgance}
 $\rho(\p)=\frac{4}{3}.$
\end{theorem}
\begin{proof}
By Proposition \ref{Prop:Contain}
\[
 \rho(\p)=\sup\left\{\frac{4k}{3k+1},\frac{4k+1}{3k+2},\frac{4k+2}{3k+2},\frac{4k+3}{3k+3} : k \geq 0\right\}
 =\frac{4}{3}.
\]
\end{proof}

\subsection{Waldschmidt Constant}

Consider the  polynomial ring  $T = \kk[x_1, x_2, x_3]$  with weights $d_i= wt(x_i)$ where $d_1  = 3$, $d_2 = 3+m$ and $d_3 = 3 + 2m$. With these weights, $\p^n$ and  $\p^{(n)}$ are  weighted homogenous ideals. For any weighted homogenous ideal $I \subseteq T$, let $\alpha(I):= \min \{ n | I_n \not = 0 \}$. Recall that the Waldschmidt  
constant is defined as 
$$
\gamma(I) = \limm~\f{\alpha( I^{(n)})}{n}.
$$
In this section we compute $\alpha(\p)/ \gamma(\p)$ and compare it with $\rho(\p)$.
We obtain similar results as in
 \cite[Theorem~1.2.1]{BH} and \cite[Lemma~2.3.2]{BH}.

\begin{theorem}
\label{waldschmidt constant}
\been
\item
\label{waldschmidt constant one}
$\alpha(\p) = 2m +6$

\item
\label{waldschmidt constant two}
$
{\displaystyle
\gamma(\p) = \begin{cases}
15/2 & \mbox{ if } m=1 \\
2m+6 & \mbox{ if } m >1.
\end{cases}
}
$
\eeen
\end{theorem}
\begin{proof} Note that $\p = (\Delta_1, \Delta_2, \Delta_3)$ where $\Delta_1$ $\Delta_2$ and $\Delta_3$ are as defined in (\ref{def of delta}). Then
$\deg(  \Delta_1) = 4m+6$, $\deg(  \Delta_2) = 3m+6$, $\deg(\Delta_3) = 2m+6$, and $\deg(f_2) = 6m+9$. Hence 
$\alpha(\p) = 2m + 6$. 

We now compute $\alpha (\p^{(n)})$. Then from (\ref{Eqn:f_2}) we get 
\beqn
\alpha( \p^{(2n)} )&= &
\begin{cases}
 n\deg(f_2) =15n   &  \mbox{ if } m=1\\
2n \deg( \Delta_3) = 2n (2m + 6) & \mbox { if } m >1
\end{cases} \\
\alpha( \p^{(2n + 1)})
&= &
\begin{cases}
15n   + 8 = n \deg(f_2) + \deg(\Delta_3)&   \mbox{ if } m=1\\
(2n+1)  \deg( \Delta_3) =(2n+1) (2m + 6 ) & \mbox { if } m >1
\end{cases}.
\eeqn
Hence 
\beqn
\gamma(\p) =  \limm~\f{\alpha( \p^{(n)})}{n} 
= \begin{cases}
15/ 2 \mbox{ if } m=1\\
2m + 6 \mbox{ if } m >1.
\end{cases}
\eeqn
 \end{proof}
%

\begin{theorem}
\label{lb for resurgence}
\beqn
1 \leq \f{\alpha(\p)}{ \gamma(\p)} \leq  \rho(\p).
\eeqn
\end{theorem}
\begin{proof}
By Theorem~\ref{waldschmidt constant}, 
\beqn
\f{\alpha(\p)}{\gamma(\p) }
=\begin{cases}
 \f{8 }{15/2} = \f{16}{15} & \mbox{ if } m=1\\
\f{2m+6}{2m+6} = 1 & \mbox{ if } m> 1.
\end{cases}
\eeqn
By Theorem~\ref{resurgance}, the result follows. 
\end{proof}

\subsection{Regularity}
In \cite{cut-kurano}, S.~D.~Cutkosky and K. Kurano studied the regularity of saturated ideals in a weighted projective space.

In this subsection  we consider the polynomial ring  $T = \kk[x_1, x_2, x_3]$  with weights $d_i= wt(x_i)$, where $d_1  = 3$, $d_2 = 3+m$ and $d_3 = 3 + 2m$. With these weights, $\p^{(n)}$ is a weighted homogenous ideal. We compute the regularity of $\p^{(n)}$ for all $n \geq 1$. 

We begin with some basic results comparing $\p^{(n)}T^{\prime}$ and $I_nT^{\prime}$.

\begin{lemma}
\label{ideals in tprime}
For all $n \geq 1$, 
\been
\item
\label{ideals in tprime one}
$\p^{(n)} T^{\prime} = I_n T^{\prime}$.

\item
\label{ideals in tprime two}
$I_{2n} \Tprime = I_2^n \Tprime$. 

\item 
\label{ideals in tprime three}
$I_{2n+1} \Tprime = I_2I_{2n} \Tprime$. 
\eeen
\end{lemma}
\begin{proof} 
Since $\J_i T^{\prime}= J_i T^{\prime}$  for $i=1,2$, we get $\I_n T^{\prime}= I_n T^{\prime}$  for all $n \geq 1$. Hence from 
Theorem~\ref{symbolic power}(\ref{symbolic power one}), $\p^{(n)} T^{\prime} = \I_n T^{\prime} = I_n T^{\prime}$.

(2) and (3) follow from (1) and (\ref{Eqn:f_2})
\end{proof}

 \begin{lemma}
 \label{reg comparision}
  For all $n \geq 1$, 
 $\reg( T/ \p^{(n)}  T )  =  \reg(\Tprime/ I_n ) $. 
  \end{lemma}
  \begin{proof} As $x_1$ is a nonzerodivisor on $T/ \p^{(n)}$ and $T / I_n T$, 
  \beqn
   \reg  \left(  \f{\Tprime } { I_{n}     }   \right) 
   &=& \reg  \left(  \f{T } { I_{n}     }   \right)  \\
&=& \reg    \left(  \f{T }  { I_{n} + (x_1) }  \right)  - 2  \hspace{.2in} \mbox{\cite[Remark~4.1]{chardin}}\\
 &=& \reg \left(  \f{T  }{ \p^{(n)} + (x_1)}  \right) - 2   \\
 &=& \reg \left(  \f{T}  { \p^{(n)}} \right)    \hspace{.1in}   \mbox{\cite[Remark~4.1]{chardin}}. 
  \eeqn

 Let $F_{\bullet}$ be a minimal free resolution of $T^{\prime} / \p^{(n)} T^{\prime}$. 
  Since $T$ is a free $T^{\prime}$-module,    $F_{\bullet} \otimes_{T^{\prime}} T$ is a minimal free resolution of 
  $T^{\prime} / \p^{(n)} T^{\prime} \otimes_{T^{\prime}} T \cong T / \p^{(n)}T$.  Hence $\reg(T/ \p^{(n)}T )= \reg(T^{\prime} / 
  \p^{(n)} T^{\prime})$. By Lemma~\ref{ideals in tprime}(1), $\p^{(n)} T^{\prime} = I_n T^{\prime}$. 
 \end{proof}
  
It follows from Lemma~\ref{reg comparision}, that in order to compute the regularity of $ T/ \p^{(n)}$,  it is enough to compute the regularity of $  {\Tprime}/{I_{n}  } $.

\begin{lemma}
\label{regularity modulo x_2^2}
Let  $n \geq 1$.  Then 
\beqn
\reg\left( \f{T^{\prime}}{I_{n} + ( x_2^2) } \right) 
= 
\begin{cases}
  \f{3d_3}{2} n + 2 d_2-2 & \mbox{ if  $n$ is even}\\
\f{3d_3}{2} n + d_2 + \f{d_3}{2} -2  & \mbox{ if  $n$ is odd}    . 
\end{cases}
\eeqn
\end{lemma}
\begin{proof}
Let  $n=2r$, where $r \geq 1$.  Then by Lemma~\ref{ideals in tprime}(\ref{ideals in tprime two}), 
\beq
\label{I_n even}
           (I_{2r}  +  (x_2^2)  ) \Tprime
   =        I_{2}^r   \Tprime+  (x_2^2)   \Tprime
   =    (x_2^4, x_2^3 x_3, x_2^2x_3^2, x_3^3)^r  \Tprime+ (x_2^2) \Tprime 
   = (x_2^2, x_3^{3r}) \Tprime .
           \eeq
Hence
\beqn
\reg\left( \f{\Tprime}{I_{2r} +  (x_2^2)} \right) =  3rd_3 + 2d_2 -2 = \f{3d_3}{2} n + 2 d_2-2. 
\eeqn 
          
 Let  $n = 2r-1$, where $r \geq 1$.  Then 
by Lemma~\ref{ideals in tprime}(\ref{ideals in tprime three}) we get 
\beqn
         (I_{2r-1} +  (x_2^2)) \Tprime
&=& (I_1 +  (x_2^2)) (I_{2(r-1)} + ( x_2^2))  \Tprime+ (x_2^2)\Tprime\\
&=&  ( x_2^2, x_2 x_3, x_3^{ 2}) ( x_2^2, x_3^{3(r-1)}) \Tprime + (x_2^2 ) \Tprime 
         \hspace{.5in} \mbox{[by  (\ref{I_n even})]}\\
&=& ( x_2^2, x_2 x_3^{3r-2}, x_3^{ 3r-1}) \Tprime.
\eeqn
By Hilbert-Burch theorem the minimal free resolution  of $(I_{2r-1} + (x_2^2))T^{\prime}$ is
$$
\xymatrix@C=18pt{
0\ar[r]  &  {\begin{array}{c}  
              T^{\prime}[ - 2d_2 - (3r-2)d_3 ] \\ 
                \oplus  \\
              T^{\prime}[ -d_2  - (3r-1) d_3]\end{array}}
\ar[rrr]^(0.5){  \left( 
\begin{array}{cc}
x_3^{3r-2} & 0\\
-x_2           & -x_3\\
0                & x_2
\end{array} \right)}
&&& { \begin{array}{c} 
T^{\prime}[-2d_2]  \\ \oplus \\ 
T^{\prime}[ - d_2 - (3r-2)d_3 ]  \\  \oplus \\ 
T^{\prime}[- (3r-1) d_3]\end{array}} 
\ar[r] & T^{\prime}
\ar[r] &  {\displaystyle \f{T^{\prime}}{I_{2r-1} + (x_2^2)}}
\ar[r] & 0.
}
$$
This gives
\beqn
\reg \left(\f{T^{\prime}}{I_{2r-1} + (x_2^2)} \right) 
=   (3r-1)d_3  + d_2   -2
=  \f{3 d_3}{2} n + d_2 + \f{d_3}{2} -2. 
\eeqn
\end{proof}

\begin{lemma}
\label{lemma mod x_3}
For all $n \geq 1$, 
\beqn
\reg\left( \f{T^{\prime}}{I_{2n} + ( x_3^3) } \right) 
= 2d_2(2n) -2d_2+ 3d_3-2.
\eeqn
\end{lemma}
\begin{proof}
By  Lemma~\ref{ideals in tprime}(\ref{ideals in tprime two}) we get
\beqn
    I_{2n} + (x_3^3) 
= I_2^n + (x_3^3) 
= (x_2, x_3)^{4n} + (x_3^3)
= (x_2^{4n}, x_2^{4n-1}x_3, x_2^{4n-2} x_3^2, x_3^3).
\eeqn
Hence by Hilbert-Burch theorem the minimal free resolution of $ I_{2n} + (x_3^3)$ is 
\small
$$
\xymatrix@C=25pt{
0\ar[r]  &  {\begin{array}{c}  
T^{\prime}[ - (4n-1 )d_2 -2 d_3  ] \\ \oplus  \\T^{\prime}[ -4n d_2 - d_3]
\\ \oplus \\
T^{\prime}[-(4n-2) d_2 - 3 d_3]
\end{array}}
\ar[rrr]^(0.5){  \left( 
\begin{array}{ccc}
0 & x_3      &0\\
x_3          & -x_2 & 0\\
 -x_2 & 0 &- x_3\\
 0& 0& x_2^{4n-2}
\end{array} \right)}
&&& { \begin{array}{c} T^{\prime}[-4n d_2]  \\ \oplus \\ T^{\prime}[ -(4n-1) d_2  - d_3]  \\  \oplus \\ T^{\prime}[- (4n-2)  d_2 - 2 d_3]\\
\oplus \\T^{\prime}[ -3 d_3]
\end{array}} 
\ar[r] & T^{\prime}
\ar[r] &  {\displaystyle \f{T^{\prime}}{I_{2n} + (x_3^3)}}
\ar[r] & 0.
}
$$
\normalsize
This gives $\reg( T/ I_{2n} + (x_3^3)) = 3 d_3  + (4n-2) d_2 - 2 = 2d_2(2n) -2d_2+ 3d_3-2
$. 
\end{proof}

\begin{proposition}
\label{regularity even}
Let  $n \geq 1$. Then
\beqn
  \reg \left( \f{\Tprime}{I_{2n} }\right) 
&=& \begin{cases} 
(2d_2)(2n) - 2d_2 + 3d_3 -2  & \mbox{ if } m=1,\\
 \f{3d_3}{2} (2n) + 2d_2-2& \mbox{ if }  m \geq 2   . 
\end{cases}
\eeqn
\end{proposition}
\begin{proof} For all $n \geq 1$, the sequence
\normalsize
$$
\xymatrix@C=20pt{
0 \ar[r] &{\displaystyle \f{\Tprime}{I_{2n-2} } [-3d_3]}
\ar[r]^{.x_3^3}  & {\displaystyle \f{\Tprime}{I_{2n}}}
\ar[r] & {\displaystyle \f{\Tprime}{I_{2n} + (x_3^3)}}
\ar[r] & 0\\
}
$$
is exact by Proposition~\ref{symbolic ideal quotients}. Hence 
\beqn
   \reg\left( \f{\Tprime}{I_{2n}}\right)
&=& \max\left\{    \reg\left( \f{\Tprime}{I_{2n-2}}\right) + 3 d_3 , 
                            \reg\left( \f{\Tprime}{I_{2n}   + (x_3^3) }\right)  \right\}\\
&=& \max\left\{    \reg\left( \f{\Tprime}{I_{2n-4}}\right) + 6 d_3 , 
                           \reg\left( \f{\Tprime}{I_{2n-2}   + (x_3^3) }  \right) + 3 d_3 , 
                           \reg\left( \f{\Tprime}{I_{2n}   + (x_3^3) }\right)   \right\}\\
&=& \vdots\\
&=& \max\left\{   \reg\left( \left. \f{\Tprime}{I_{2n-2i} + (x_3^3)}\right) + 3i d_3 \right| 
                                                                                      i=0, \ldots, n-1  \right\}\\
&=& \max\left\{ \left. 2d_2(2n-2i)   -2d_2 + 3d_3  -2         + 3i d_3\right| i=0, \ldots, n-1  \right\} \mbox{[ by Lemma~\ref{lemma mod x_3}]}\\
&=& \max\left\{ \left. 4nd_2 - 2d_2 + 3d_3 -2 + i(-4d_2 + 3d_3)
\right| i=0, \ldots, n-1  \right\} \\
&=& \begin{cases} 
(2d_2)(2n) - 2d_2 + 3d_3 -2  & \mbox{ if } m=1,\\
(2d_2)(2n) - 2d_2 + 3d_3 -2 + (n-1)(-4d_2 + 3d_3)& \mbox{ if }  m \geq 2   . 
\end{cases}\\
&=& \begin{cases} 
(2d_2)(2n) - 2d_2 + 3d_3 -2  & \mbox{ if } m=1,\\
 \f{3d_3}{2} (2n) + 2d_2-2& \mbox{ if }  m \geq 2   . 
\end{cases}
\eeqn
\end{proof}

\begin{proposition}
\label{regularity odd}
 Let  $n \geq 1$. Then
\beqn
  \reg \left( \f{\Tprime}{I_{2n+1} }\right) 
= \begin{cases}
(2d_2) (2n+1) - 2d_2 + 3d_3 -2  & \mbox{ if } m=1\\
\f{3d_3}{2} (2n+1) + 4d_2 -\f{3d_3}{2} -2 & \mbox{ if } m=2 \\
 \f{3d_3}{2}  (2n+1) +d_2 + \f{d_3}{2} -2     & \mbox{ if } m  \geq 3
\end{cases}.
\eeqn
\end{proposition}
\begin{proof}
 For all $n \geq 1$, the sequence
\normalsize
$$
\xymatrix@C=20pt{
0 \ar[r] &{\displaystyle \f{\Tprime}{I_{2n}}[-2d_2] }
\ar[r]^{.x_2^2}  & {\displaystyle \f{\Tprime}{I_{2n+1}}}
\ar[r] & {\displaystyle \f{\Tprime}{I_{2n+1} + (x_2^2)}}
\ar[r] & 0\\
}
$$
is exact by Proposition~\ref{symbolic ideal quotients}. Hence
\beq
\label{regularity odd and even}
\reg\left( \f{\Tprime}{I_{2n+1}}\right) 
= \max\left\{    \reg\left( \f{\Tprime}{I_{2n}}\right) + 2 d_2 , \reg\left( \f{\Tprime}{I_{2n+1}   + (x_2^2) }\right)  \right\}.
\eeq
Using  Proposition~\ref{regularity even} and  Lemma~\ref{regularity modulo x_2^2} in (\ref{regularity odd and even}) we get
\beqn
  \reg\left( \f{\Tprime}{I_{2n+1}}\right) 
&=&\begin{cases}
\max\left\{(2d_2) (2n+1) - 2d_2 + 3d_3 -2,
               \f{3d_3}{2}  (2n+1) +d_2 + \f{d_3}{2} -2    \right\} & \mbox{ if } m=1\\
\max\left\{ \f{3d_3}{2} (2n+1) + 4d_2 -\f{3d_3}{2} -2,
              \f{3d_3}{2}  (2n+1) +d_2 + \f{d_3}{2} -2     \right\}   & \mbox{ if } m\geq 2 
   \end{cases} \\
&=&
\begin{cases}
(2d_2) (2n+1) - 2d_2 + 3d_3 -2  & \mbox{ if } m=1\\
\f{3d_3}{2} (2n+1) + 4d_2 -\f{3d_3}{2} -2 & \mbox{ if } m=2 \\
 \f{3d_3}{2}  (2n+1) +d_2 + \f{d_3}{2} -2     & \mbox{ if } m  \geq 3
\end{cases}.
\eeqn   
\end{proof}

\begin{theorem}
\label{final regularity}
\been
\item
\label{final regularity one}
 $\reg (T/ \p)= d_2 + 2d_3-2$. 

\item 
\label{final regularity two}
Let $n \geq 2$. 
\been
\item If $m=1$, then $\reg (T/ \p^{(n)}) = (2d_2)n -2d_2 + 3d_3-2$. 

\item
If $m=2$, then  ${
\displaystyle 
  \reg \left( \f{T}{ \p^{(n) } } \right) 
= \begin{cases}
   \f{3d_3}{2}n + 4d_2 -\f{3d_3}{2} -2& \mbox{ if  $n$ is odd},\\
   \f{3d_3}{2}n   + 2d_2-2  & \mbox{ if  $n$ is even}   . 
\end{cases}
}
$.

\item
If $m \geq 3$, then  ${\displaystyle 
  \reg \left( \f{T}{ \p^{(n) } } \right) 
= \begin{cases}
  \f{3d_3}{2}n + d_2 + \f{d_3}{2}-2& \mbox{ if  $n$ is odd}\\
  \f{3d_3}{2}n   + 2d_2-2 & \mbox{ if  $n$ is even}  . 
\end{cases}
}$.
\eeen
\eeen
In particular, ${ \displaystyle \limnn  \reg(    (\p^{n})^{sat}   ) /n=\frac{3 e(T/ \p)}{2}  + 3m}$.
\end{theorem}
\begin{proof}
By Lemma~\ref{reg comparision},   $\reg (T/ \p^{(n)} )= \reg (\Tprime/  I_n) $. 
Since $I_1 + (x_2^2) = I_1$, (\ref{final regularity one})  from Lemma~\ref{regularity modulo x_2^2}. 
(\ref{final regularity two}) follows from Proposition~\ref{regularity even} and Proposition~\ref{regularity odd}. 

Finally, ${ \displaystyle \limnn  \reg(    (\p^{n})^{sat}   ) /n= \frac{3 d_3}{2} = \frac{3( e(T/ \p) + 2m )}{2} = \frac{3 e(T/ \p)}{2}  + 3m}$
\end{proof}

\begin{theorem}
\label{ub for resurgence}
$\rho(\p) \leq \reg(\p) / \gamma(\p).$
\end{theorem}
\begin{proof} By Theorem~\ref{final regularity} and Theorem~\ref{waldschmidt constant}(\ref{waldschmidt constant two})
\beqn
\f{\reg(\p) }{ { \gamma} (\p)}
&=& \begin{cases}
{\displaystyle \f{13}{{15/2}} = \f{26}{15} \geq \f{4}{3} = \rho (\p)}  & \mbox{ if } m=1 \vspace{.1in} \\
{\displaystyle \f{3m + (9/2)}  {2m + 6}  \geq \f{4}{3} = \rho (\p) } & \mbox{ if } m \geq 2. 
\end{cases}
\eeqn
\end{proof}


\end{document}